\documentclass[12pt]{amsart}
\usepackage[margin=3cm]{geometry}

\usepackage{microtype} 

\usepackage{amsthm,amsmath,amssymb,amsfonts}
\usepackage{mathtools}
\usepackage{mathrsfs}
\usepackage[skip=4pt plus1pt, indent=10pt]{parskip} 
\usepackage{enumerate}
\usepackage{xcolor}
\usepackage{tikz}
\usepackage[colorlinks,bookmarks, anchorcolor=blue, citecolor=blue, linkcolor=blue,urlcolor=blue, bookmarksopenlevel=1]{hyperref}
\usepackage{tikz-cd}
\usepackage{appendix}
\usepackage{afterpage}
\usepackage{makecell}
\usepackage[new]{old-arrows}
\usepackage[normalem]{ulem}

\usepackage[style=alphabetic,sorting=nyt,firstinits=true,maxnames=5,maxalphanames=5,backend=bibtex, doi=false,
isbn=false,
url=false,
eprint=false
]{biblatex}
\makeatletter
\newcommand{\addresseshere}{%
  \enddoc@text\let\enddoc@text\relax
}
\makeatother

\numberwithin{equation}{section}

\newtheorem{theorem}{Theorem}[section]
\newtheorem{proposition}[theorem]{Proposition}
\newtheorem{lemma}[theorem]{Lemma}
\newtheorem{corollary}[theorem]{Corollary}

\theoremstyle{definition}

\newtheorem{remark}[theorem]{Remark}
\newtheorem{example}[theorem]{Example}

\newcommand{\rk}{\mathrm{rk}}
\newcommand{\vol}{\mathrm{vol}}
\newcommand{\discel}{\mu}
\newcommand{\discelone}{\alpha_1}
\newcommand{\disceltwo}{\alpha_2}
\newcommand{\discelgen}{\alpha_i}
\newcommand{\weil}[1]{\rho_{#1}^*}
\newcommand{\ZZ}{\mathbb{Z}}
\newcommand{\RR}{\mathbb{R}}
\newcommand{\QQ}{\mathbb{Q}}
\newcommand{\CC}{\mathbb{C}}
\newcommand{\HH}{\mathbb{H}}

\newcommand{\Sp}{\mathrm{Sp}}
\newcommand{\Mp}{\mathrm{Mp}}
\newcommand{\bigO}{\mathrm{O}}
\newcommand{\Pic}{\mathrm{Pic}}

\newcommand{\KMtheta}[1]{\Theta_{\operatorname{KM},#1}}
\newcommand{\KMschwartz}{\varphi_{\operatorname{KM}}}
\newcommand{\KMschwartzg}[1]{\varphi_{\operatorname{KM},#1}}

\newcommand{\trg}[3]{\mathrm{tr}_{#1,#2/#3}}
\DeclareMathOperator{\trace}{tr}
\newcommand{\Kmax}{K_{\mathrm{max}}}
\newcommand{\genus}{g}
\newcommand{\halfint}{\mathcal{S}}

\begin{document}\baselineskip=15pt

\begin{center}
\title[Extremal divisors on moduli of K3 surfaces]{Extremal divisors on moduli spaces of K3 surfaces}

\author{Ignacio Barros}
\address{\parbox{0.9\textwidth}{
Department of Mathematics\\[1pt]
Universiteit Antwerpen\\[1pt]
Middelheimlaan 1, 2020 Antwerpen, Belgium
\vspace{1mm}}}
\email{{ignacio.barros@uantwerpen.be}}

\author{Laure Flapan}
\address{\parbox{0.9\textwidth}{
Department of Mathematics\\[1pt]
Michigan State University\\[1pt]
619 Red Cedar Road, East Lansing, MI 48824, USA
\vspace{1mm}}}
\email{{flapanla@msu.edu}}

\author{Riccardo Zuffetti}
\address{\parbox{0.9\textwidth}{
Fachbereich Mathematik\\[1pt]
Technische Universität Darmstadt\\[1pt]
Schlossgartenstraße 7, D–64289 Darmstadt, Germany
\vspace{1mm}}}
\email{{zuffetti@mathematik.tu-darmstadt.de}}

\thanks{I.B. was supported by the Research Foundation – Flanders (FWO) within the framework of the Odysseus program project number G0D9323N and by the Deutsche Forschungsgemeinschaft (DFG, German Research Foundation) - Project-ID 491392403 – TRR 358. 
L.F.~was supported by NSF grant DMS-2200800.
R.Z.~was supported by Deutsche Forschungsgemeinschaft (DFG, German Research Foundation) through the Collaborative Research Centre TRR 326 \textit{Geometry and Arithmetic of Uniformized Structures}, project number 444845124.}

\maketitle
\end{center}

\vspace{-0.7cm} 

\begin{abstract}
We establish criteria for when Noether--Lefschetz divisors generate an extremal ray in the cone of pseudoeffective divisors of an orthogonal modular variety.
In particular, we exhibit many extremal rays of the cone of pseudoeffective divisors on any moduli space~$\mathcal{F}_{2d}$ of quasi-polarized K3 surfaces of degree $d$, as well as on any normal projective $\mathbb{Q}$-factorial compactification $\overline{\mathcal{F}}_{2d}$ of $\mathcal{F}_{2d}$ lying over the Baily--Borel compactification.
\end{abstract}

\setcounter{tocdepth}{1} 
\tableofcontents

\vspace{-0.5cm} 

\section{Introduction}

An important, yet notoriously difficult invariant for understanding the birational geometry of an algebraic variety $X$ is its cone of pseudoeffective divisors $\overline{{\rm Eff}}(X)$. In the case when $X$ is a moduli space, these cones have been extensively studied, yet the available results are still quite limited. For the moduli space of stable curves $\overline{\mathcal{M}}_g$, there is only a full description of the generators of the effective cone when $g\leq 3$, see \cite{Rul01}. In higher genus, extremal divisors in ${\rm{Pic}}\left(\overline{\mathcal{M}}_g\right)_{\mathbb{R}}$ that are supported on the interior of~$\overline{\mathcal{M}}_g$ are known only when $g\leq 11$, see \cites{HM90, CR91, Tan98, FP05, CHS08}. Similarly, for the (suitably compactified) moduli space of principally polarized abelian varieties $\overline{\mathcal{A}}_g$, extremal divisors supported on the interior are known only for $g\leq 5$, see \cite{SM92, FGSMV14}.

In this paper, our main interest is when $X$ is an orthogonal modular variety, with a special focus on the case when $X$ is a moduli space $\mathcal{F}_{2d}$ of primitively (quasi)-polarized K3 surfaces $(S,H)$ of degree $H^2=2d$. The study of $\overline{{\rm Eff}}(\mathcal{F}_{2d})$ was initiated in \cites{Pet15, BLMM17, BM19} (see also \cite{zuffetti-cod2, zuffetti-equidistribution, BBFW24}), yet the only known examples of extremal rays are for $d=1,2$. The method used in these examples is to compare the GIT and Baily--Borel models of $\mathcal{F}_{2d}$, see \cites{Sha80, Loo86, LO18, LO19, ADL23}, and relies on having a concrete birational model for $\mathcal{F}_{2d}$ such as those in low $d$ constructed in \cites{Muk88, Muk92, Muk06, Muk10, Muk16}. Recall however that for large $d$, the moduli space $\mathcal{F}_{2d}$ is of general type for large $d$, see \cite{GHS07}.

The main results of our paper give a criterion for the extremality of the irreducible components of Noether--Lefschetz divisors on $\mathcal{F}_{2d}$, which allow us to produce extremal rays of $\overline{{\rm Eff}}(\mathcal{F}_{2d})$ for all $d$, with the number of such rays increasing with $d$. Our approach relies on the relationship between orthogonal modular varieties and automorphic forms and has the advantage that we can study the positivity of Noether--Lefschetz divisors for arbitrarily large $d$, when no geometric models are available. Our results recover the previously known examples of extremal divisors on $\mathcal{F}_{2d}$ when $d=1,2$. Moreover, our criterion produces extremal divisors for any orthogonal modular variety, including moduli spaces of (lattice) polarized hyperk\"{a}hler manifolds.

For us, the convex cone $\overline{{\rm Eff}}(X)$ is the closure in $\mathrm{Pic}_{\mathbb{R}}(X)$ of the cone of effective $\mathbb{R}$-divisors on $X$ up to linear equivalence.
When $X$ is normal, projective, and $\mathbb{Q}$-factorial, and ${\rm{Pic}}^0(X)$ is trivial (for instance if $b_1(X)=0$), this definition of $\overline{{\rm Eff}}(X)$ coincides with the usual one---the closure in ${\rm{NS}}(X)_{\mathbb{R}}$ of the cone of effective divisors.

Noether--Lefschetz divisors $D_{h,a}\subset\mathcal{F}_{2d}$ are a
natural collection of modular divisors on $\mathcal{F}_{2d}$, defined as the reduced divisors given by the closure of the locus of quasi-polarized K3 surfaces $(S,H)$ with a class $\beta\in {\rm{Pic}}(S)$ not proportional to $H$ such that $\beta^2=2h-2$ and $\beta\cdot H=a$. One can always assume $0\leq a<2d$, and by the Hodge Index Theorem, the determinant of the intersection matrix of $\beta$ and $H$ must be negative.

A special consequence of the main results of the present paper is the following extremality criterion, which works for any moduli space~$\mathcal{F}_{2d}$.

\begin{theorem}
\label{intro:thm:main_K3}
 If the Noether--Lefschetz divisor $D_{h,a}\subset\mathcal{F}_{2d}$ satisfies
\[
\left(\frac{a^2}{4d}-(h-1)\right)^2\cdot\frac{8d}{\gcd(a,2d)}<1,
\]
then 
\begin{enumerate}
    \item 
Every irreducible component of $D_{h,a}$ is extremal in $\overline{\rm{Eff}}(\mathcal{F}_{2d})$. \item Every irreducible component of the closure 
$\overline{D}_{h,a}$ in $\overline{\mathcal{F}}_{2d}$ 
 is extremal in $\overline{\rm{Eff}}\big(\overline{\mathcal{F}}_{2d}\big)$, where $\overline{\mathcal{F}}_{2d}\longrightarrow \mathcal{F}_{2d}^{BB}$ is any  normal projective $\mathbb{Q}$-factorial variety compactifying~$\mathcal{F}_{2d}$ over the Baily--Borel model $\mathcal{F}_{2d}^{BB}$.
 \end{enumerate}
\end{theorem}

\begin{example}
By Theorem \ref{intro:thm:main_K3} the irreducible unigonal divisor $D_{1,1}\subset \mathcal{F}_{2d}$ is extremal for all $d$. The same holds (see Theorems \ref{thm: extremality} and \ref{thm: numerical bounds}) for the components of the hyperelliptic divisor $D_{1,2}\subset \mathcal{F}_{2d}$ whose general point consists of a polarized K3 surfaces~$(S,H)$ where the linear system $\left|H\right|$ gives a $2$-to-$1$ map $S\longrightarrow \mathbb{F}_a\subset \mathbb{P}^{d+1}$ to a rational normal scroll ramified over a curve $B\in\lvert-2K_{\mathbb{F}_a}\rvert$, see \cite{Rei76}. We refer to Table \ref{tableoutmethod2:2} for more examples.
\end{example}

In fact, rather than working only with the moduli space $\mathcal{F}_{2d}$, we work over the more general setting of \emph{orthogonal modular varieties} and the question of extremality of irreducible components of {\textit{Heegner divisors}}. Theorem \ref{intro:thm:main_K3} above is a direct consequence of Theorem \ref{thm: extremality} below. It also follows from the explicit numerical bounds obtained in Theorem \ref{thm: numerical bounds}. As a consequence, one also obtains extremality criteria for Noether--Lefschetz divisors on moduli spaces of lattice-polarized K3 surfaces and hyperk\"{a}hler manifolds.

\subsection{The general setting}
Let $\Lambda$ be an even lattice of signature $(2,n)$.
The associated orthogonal modular variety
\[\mathcal{F}_\Lambda=\mathcal{D}_\Lambda/\widetilde{\rm{O}}^+(\Lambda)\]
is the quotient of one connected component~$\mathcal{D}_\Lambda$ of the Type IV Hermitian symmetric domain
\[
\left\{[Z]\in \mathbb{P}\left(V_\mathbb{C}\right)\left|\langle Z,Z\rangle=0, \langle Z,\overline{Z}\rangle>0 \right.\right\}
\]
by the discriminant kernel~$\widetilde{\rm{O}}^+(\Lambda)$ of the group of automorphisms of~$\Lambda$ fixing~$\mathcal{D}_\Lambda$.

Examples of orthogonal modular varieties of geometric nature include moduli spaces of polarized K3 surfaces, hyperk\"{a}hler varieties, (special) cubic fourfolds, and $(1,d)$-polarized abelian surfaces.
For instance, if $\Lambda$ arises as the lattice $c_1(H)^\perp\subset H^2(S, \mathbb{Z})$ associated to a primitively polarized K3 surface $(S,H)$ of degree $H^2=2d$, then $\Lambda$ is isomorphic to $U^{\oplus 2}\oplus E_8(-1)^{\oplus 2}\oplus A_1(-d)$.
The corresponding modular variety $\mathcal{F}_{\Lambda}$ is a coarse moduli space for primitively quasi-polarized K3 surfaces of degree $2d$. 

The main goal of this paper is to describe extremal rays of the cone of pseudoeffective divisors~$\overline{{\rm Eff}}(\mathcal{F}_\Lambda)\subset {\rm{Pic}}_\mathbb{R}\left(\mathcal{F}_{\Lambda}\right)$ on an orthogonal modular variety $\mathcal{F}_\Lambda$ as well as extremal rays of $\overline{{\rm Eff}}(\overline{\mathcal{F}}_\Lambda)$ on  compactifications~$\overline{\mathcal{F}}_\Lambda$ of $\mathcal{F}_\Lambda$. 

\textit{Heegner divisors} on $\mathcal{F}_\Lambda$ are $\QQ$-Cartier effective divisors arising as images of certain hyperplane arrangements via the quotient $\pi:\mathcal{D}_{\Lambda}\longrightarrow\mathcal{F}_{\Lambda}$ (see Section \ref{sec2}). They are denoted by $H_{-m, \mu}$ for $\mu$ a class in the discriminant group $D(\Lambda)$ and $m\in \mathbb{Z}-q\left(\mu\right)>0$.

\textit{Primitive Heegner divisors}
also arise as images of hyperplane arrangements via the quotient $\pi:\mathcal{D}_{\Lambda}\longrightarrow\mathcal{F}_{\Lambda}$ (see Section \ref{sec2}) with an additional primitivity assumption on the vectors defining the hyperplanes.
They are denoted by $P_{-m, \mu}$ for~$m$ and~$\mu$ as above. In the case that the lattice $\Lambda$ splits off two copies of the hyperbolic plane $U$, primitive Heegner divisors are irreducible components of Heegner divisors. A description of which primitive Heegner divisors arise as irreducible components of a given Heegner divisor~$H_{-m, \mu}$ in this case can be found in Section \ref{sec1.1:prelim}. 

Theorem \ref{intro:thm:main_K3} above about the moduli space $\mathcal{F}_{2d}$ is a consequence of the following criterion for extremality, which is a simplified form of Theorem \ref{thm: extremality2} below.
It depends on the Fourier coefficients~$c_{t,\alpha_L}(E_{k,L})$ of the Eisenstein series~$E_{k,L}$, for some sublattice~$L$ of~$\Lambda$.
The Eisenstein series~$E_{k,L}$ is an elliptic, weight ${k\coloneqq(n+1)/{2}}$ modular form, which is vector-valued with respect to the dual Weil representation of~$L$ (see Section~\ref{sec;vv Siegel mod}), and its coefficients~$c_{t,\alpha_L}(E_{k,L})$ are rational numbers, computed in \cite{BK01}.

\begin{theorem}
\label{thm: extremality}
Let $\Lambda$ be an even lattice of signature $(2,n)$ with $n\ge 4$, and let~$P$ be an irreducible component of~$P_{-m,\mu}$. Such $P$ is given by the image of a map~$\mathcal{F}_L\to\mathcal{F}_{\Lambda}$ induced by the inclusion~$\mathcal{D}_L\subset\mathcal{D}_\Lambda$, for some sublattice~$L\subset\Lambda$ of signature $(2,n-1)$.
If
\begin{equation}
\label{sec1:eq:thm_ext}
-\sum_{\substack{\alpha\in\Lambda/L\oplus L^\perp\\t\in \mathbb{Z}-q_L(\alpha_L)\\0<t\leq m}}c_{t,\alpha_L}\left(E_{\frac{n+1}{2},L}\right)<\begin{cases}1&\hbox{if }2\mu=0,\\\frac{1}{2}&\hbox{otherwise},\end{cases}
\end{equation}
where~$\alpha_L$ denotes the image of~$\alpha$ under the projection~$D(L)\oplus D(K)\to D(L)$, then
\begin{enumerate}
    \item $P$ generates an extremal ray in $\overline{{\rm{Eff}}}\left(\mathcal{F}_\Lambda\right)$.
    \item The closure $\overline{P}$ generates an extremal ray in $\overline{{\rm{Eff}}}\left(\overline{\mathcal{F}}_\Lambda\right)$ for any normal projective $\mathbb{Q}$-factorial variety $\overline{\mathcal{F}}_\Lambda$ compactifying $\mathcal{F}_{\Lambda}$ over $\mathcal{F}_{\Lambda}^{BB}$. 
\end{enumerate}
\end{theorem}

As a consequence of Theorem \ref{thm: extremality}, we obtain the following simpler criterion for determining whether all irreducible components of~$P_{-m,\mu}^\Lambda$ are extremal.

\begin{corollary}
\label{thm: numerical bounds intro} 
Let $\Lambda$ be an even lattice of signature $(2,n)$ with $n\ge 4$. For any $\mu\in D(\Lambda)$ let $d_\mu$ be its order in $D(\Lambda)$, and let $m\in \mathbb{Z}-q_\Lambda(\mu)>0$.
If
\begin{equation}\label{eq;symplin1o4}
    m^2d_\mu<\frac{1}{4},
\end{equation}
then 
\begin{enumerate}
    \item Every irreducible components of $H_{-m, \mu}^\Lambda$ is extremal in $\overline{{\rm{Eff}}}(\mathcal{F}_\Lambda)$.
    \item Every irreducible component of the closure $\overline{H}_{-m,\mu}^{\Lambda}$ is extremal in $\overline{{\rm{Eff}}}(\overline{\mathcal{F}}_\Lambda)$ for any normal projective $\mathbb{Q}$-factorial variety $\overline{\mathcal{F}}_{\Lambda}$ compactifying $\mathcal{F}_{\Lambda}$ over $\mathcal{F}_{\Lambda}^{BB}$.
    \end{enumerate}
\end{corollary}

\begin{remark}
For ease of exposition, Corollary~\ref{thm: numerical bounds intro} is stated in a simplified form. We obtained a stronger criterion, where the left-hand side of~\eqref{eq;symplin1o4} is replaced by a real-valued function depending on~$n$, $m$, $d_{\mu}$ and the discriminant of~$\Lambda$, see Theorem~\ref{thm: numerical bounds}.
\end{remark}

\subsection{Application to asymptotics of extremality}

As an application of Corollary~\ref{thm: numerical bounds intro} and Theorem~\ref{thm: numerical bounds},  we study in Section \ref{sec3:asymptotic} the asymptotic behavior of extremality, showing essentially that any Heegner divisor eventually becomes extremal as either the rank of~$\Lambda$ or the discriminant $D_\Lambda$ grows. For instance, for the K3 lattice~$\Lambda_{2d}= U^{\oplus 2}\oplus E_8(-1)^{\oplus 2}\oplus \ZZ\ell$, with~$q(\ell)=-d$ and~$\ell_*=\ell/2d$ the standard generator of~$D(\Lambda_{2d})$, we establish:
\begin{corollary}
    Let $m\in \mathbb{Q}_{> 0}$ and $d_\mu\in \mathbb{Z}_{> 0}$ fixed. If $d$ is large enough, then for all~$0\leq a<2d$ with $2d/\gcd(a,2d)=d_{\mu}$ the divisor $P_{-m,a\ell_*}$ is either empty or extremal in~$\overline{{\rm{Eff}}}(\mathcal{F}_{2d})$ and its closure extremal in~$\overline{{\rm{Eff}}}(\overline{\mathcal{F}}_{2d})$.
\end{corollary}

\subsection{Application to cones of Heegner divisors}
Given that Heegner divisors form a special class of effective divisors on $\mathcal{F}_\Lambda$, it is natural to consider: 
\begin{itemize}
\item[(A)] The $\mathbb{Q}$-subspace $\Pic_\QQ^{NL}(\mathcal{F}_\Lambda)\subset \Pic_\QQ(\mathcal{F}_\Lambda)$ generated by Heegner divisors.
\item[(B)]The subcone ${\rm{Eff}}^{NL}(\mathcal{F}_\Lambda)\subset\overline{{\rm Eff}}(\mathcal{F}_\Lambda)$ generated
by irreducible components of Heegner divisors.
\end{itemize}
If~$n\ge 3$ and $\Lambda$ splits off two copies of the hyperbolic plane, then one has equality~$\Pic_\QQ^{NL}(\mathcal{F}_\Lambda)=\Pic_\QQ(\mathcal{F}_\Lambda)$  in the containment (A) \cite{BLMM17} and the NL-cone~${\rm{Eff}}^{NL}(\mathcal{F}_\Lambda)$ in (B) is polyhedral \cite{BM19}.
The extremal rays of the NL-cone ${\rm Eff}^{NL}(\mathcal{F}_\Lambda)$ can be computed explicitly for any such $\Lambda$, see \cite{BBFW24} and the examples therein. See also~\cite{zuffetti-cod2, zuffetti-equidistribution} for generalizatons to higher codimension Heegner cycles. 

It is an open question whether or not $\overline{{\rm Eff}}(\mathcal{F}_\Lambda)$ is polyhedral and if the containment
\begin{equation}\label{eq: cone containment}
{\rm Eff}^{NL}(\mathcal{F}_\Lambda)\subset\overline{{\rm Eff}}(\mathcal{F}_\Lambda)
\end{equation}
is in fact an equality (e.g. see \cite{Pet15, BM19}). For instance, there is no value of $d$ for which is known whether one has equality in \eqref{eq: cone containment} for the moduli space of quasi-polarized K3 surfaces $\mathcal{F}_{2d}$.

As an application of Theorem \ref{thm: extremality} we are able to exhibit a geometric example where the equality ${\rm Eff}_{\mathbb{R}}^{NL}(\mathcal{F}_\Lambda)=\overline{{\rm Eff}}(\mathcal{F}_\Lambda)$ in the containment \eqref{eq: cone containment} does hold. To our knowledge, this is the first (non-trivial) example where either equality or non-equality is determined. 

Let $\Lambda=U^{\oplus 2}\oplus A_{1}(-1)\oplus A_1(-3)$. Then 
 $\mathcal{F}_\Lambda$ is a partial compactification of the moduli space $\mathcal{M}_{{\rm{Kum}}_{2},2}^1$ parameterizing (quasi)-polarized hyperk\"{a}hler fourfolds of Kummer type with split polarization of degree two \cite[Lemma 4.7]{BBFW24}.
 Using \cite{Bru02b, weilrep} one has that $\dim {\rm Pic}_\mathbb{R}\left(\mathcal{F}_\Lambda\right)=2$.
 The criterion of Theorem~\ref{thm: extremality} establishes (in fact the particular criterion~$m^2d_\mu<\frac{1}{4}$ of Corollary~\ref{thm: numerical bounds intro} suffices) two primitive Heegner divisors that are extremal in~$\overline{{\rm Eff}}(\mathcal{F}_\Lambda)$. We thus conclude the following. 
 \begin{corollary}\label{prop: cone equality}
The orthogonal modular variety $\mathcal{F}_\Lambda$ with $\Lambda=U^{\oplus 2}\oplus A_{1}(-1)\oplus A_1(-3)$ partially compactifying the moduli space $\mathcal{M}_{{\rm{Kum}}_{2},2}^1$ parameterizing (quasi)-polarized hyperk\"{a}hler fourfolds of Kummer type with split polarization of degree two satisfies
 \[{\rm Eff}^{NL}(\mathcal{F}_\Lambda)=\overline{{\rm Eff}}(\mathcal{F}_\Lambda).\]
 That is, every effective divisor on this $\mathcal{F}_\Lambda$ is of Noether--Lefschetz type.
 \end{corollary}

\subsection{Strategy to deduce extremality}\label{sec;strategy intro}
To prove Theorem \ref{thm: extremality}, we establish a criterion for extremality of primitive Heegner divisors based on the extremality criterion given in \cite[Lemma 4.1]{CC14} for rays in the pseudo-effective cone of a projective variety; see Lemma~\ref{lem: cc lem}.
More precisely, for $\mathcal{F}_\Lambda^{BB}$ the Baily--Borel compactification of~$\mathcal{F}_\Lambda$ and~$\overline{\mathcal{F}}_\Lambda\rightarrow \mathcal{F}_\Lambda^{BB}$ a $\mathbb{Q}$-factorial compactification of $\mathcal{F}_\Lambda$ over $\mathcal{F}_\Lambda^{BB}$, we show that the extremality of an irreducible effective divisor $D$ in~$\overline{{\rm Eff}}(\mathcal{F}_\Lambda)$ (and of~$\overline{D}$ in~$\overline{{\rm Eff}}(\overline{\mathcal{F}}_\Lambda)$) follows from the existence of a moving curve~$[C]\in N_1\left(\overline{\mathcal{F}}_\Lambda\right)$ on the closure $\overline{D}$ of $D$ in $\overline{\mathcal{F}}_\Lambda$ such that 
  \begin{equation}\label{eq;crit of extrem}
      C\cap \partial \overline{\mathcal{F}}_\Lambda=\emptyset  \quad\text{ and }\quad \overline{D}\cdot \left[C\right]<0.
  \end{equation}

Suppose for simplicity that~$\Lambda$ splits off \textbf{two hyperbolic planes}, so that every~$P_{-m, \mu}$ is irreducible. This assumption will be dropped in the main body of the paper.
In order to pinpoint a moving curve $C$ as above in the case of~$D=P_{-m, \mu}$, we consider an orthogonal modular variety $\mathcal{F}_L$ and a finite map
  \[
  \varphi\colon \mathcal{F}_L\rightarrow \mathcal{F}_\Lambda
  \]
induced by some sublattice $L\subset \Lambda$ of signature~$(2,n-1)$ chosen depending on the parameters $m$ and~$\mu$.
The image of~$\varphi$ is the primitive Heegner divisor~$P_{-m, \mu}$.
The moving curve~$C$ is then chosen as the pushfoward under~$\varphi$ of the curve class~$\lambda_L^{n-2}$ on~$\mathcal{F}_L$.
In this case, the criterion~\eqref{eq;crit of extrem} boils down to proving that
\begin{equation}\label{eq;deg vol negative}
\deg (\varphi^* P_{-m,\mu})<0
\qquad
\text{or equivalently}
\qquad
\vol(P_{-m,\mu}^2)<0.
\end{equation}
In the above, the degree is with respect to the hyperplane class on the Baily--Borel compactification $\mathcal{F}_L^{BB}\subset \mathbb{P}^N$, and
\[
\vol(P_{-m,\mu}^2)=
\int_{\mathcal{F}_\Lambda}
[P_{-m,\mu}]\wedge[P_{-m,\mu}]\wedge\omega^{n-2}
\]
is the volume of the self-intersection of~$P_{-m,\mu}$, where~$\omega$ is the Kähler form on~$\mathcal{F}_\Lambda$ representing the Hodge line bundle~$\lambda_\Lambda$.

We provide two different ways of computing the quantities in~\eqref{eq;deg vol negative}. The first one is based on \emph{pulling back Heegner divisors}, see Theorem \ref{thm: extremality2} and Corollary \ref{sec3:coro:extremality3}. The second one consists of \emph{pulling back Siegel Eisenstein series}, see Lemma~\ref{lemma:fromdegtovol} and Proposition~\ref{prop:volprHeegner Siegel}. 
We implemented both methods in SageMath \cite{sagemath} using the WeilRep package \cite{weilrep}. The program, called \texttt{extremal\_rays}, is available on the authors' websites together with a README file with more details on the computation~\cite{barrosflapanzuffetti-prog}.

We now illustrate the idea behind both methods.

\subsection*{Pulling back Heegner divisors}
By~\cite{kudla-algcycles} the pullback of a Heegner divisor on~$\mathcal{F}_\Lambda$ under~$\varphi$ is a combination of Heegner divisors on~$\mathcal{F}_L$.
We rephrase the cited result, originally stated over the adeles, to our lattice-coset setup in Proposition~\ref{prop:pulbackheegner}.
This is implied by the arithmetic properties of the \emph{Kudla--Millson theta functions}, which are modular forms with values in spaces of differential forms on the given modular variety, first introduced in~\cite{kudlamillson-harmonicI, kudlamillson-harmonicII}; see Section~\ref{sec;genseries} for details.

This pullback formula implies that~$\deg (\varphi^* P_{-m,\mu})$ is a linear combination of degrees of Heegner divisors on~$\mathcal{F}_L$, with linear coefficients given by the Fourier coefficients of certain theta functions attached to negative definite lattices.
The degrees of Heegner divisors on~$\mathcal{F}_L$ are, up to the volume~$\vol(\mathcal{F}_L)$, the Fourier coefficients of a normalized elliptic Eisenstein series.
Thanks to the explicit formulas for the latter, proved in~\cite{BK01}, we obtain a formula for~$\deg (\varphi^* P_{-m,\mu})$ from which we deduce Theorem~\ref{thm: extremality}.

\subsection*{Pulling back Siegel Eisenstein series}
The cohomology class~$[\KMtheta{\genus}]$ of the genus~$\genus$ Kudla--Millson theta function~$\KMtheta{\genus}$ on~$\mathcal{F}_\Lambda$ is a genus~$\genus$ Siegel modular form with values in~$H^{2\genus}(\mathcal{F}_\Lambda,\CC)$.
Its Fourier expansion provides the generating series of the so-called \emph{special cycles}.
If~$g=1$, these are the Heegner divisors considered above.

The cup product of two Heegner divisors on~$\mathcal{F}_\Lambda$ is a Fourier coefficient of the class~${[\KMtheta{1}\wedge\KMtheta{1}]}$.
Hence, the volume~$\vol(H_{m,\mu}^2)$ is a Fourier coefficient of
\begin{equation}\label{eq;doublthetaKM}
\int_{\mathcal{F}_\Lambda}\KMtheta{1}\wedge\KMtheta{1}\wedge\omega^{n-2},
\end{equation}
the latter considered as a modular form on~$\HH\times\HH$.

By~\cite{kudla-algcycles}, we know that~$\KMtheta{1}(\tau_1)\wedge\KMtheta{1}(\tau_2)=\KMtheta{2}\big(\begin{smallmatrix}
    \tau_1 & 0\\
    0 & \tau_2
\end{smallmatrix}\big)$, where~$\HH\times\HH$ is included in the genus~$2$ Siegel upper half-space~$\HH_2$ as
\begin{equation}\label{eq;stdembHH}
\HH\times\HH\longhookrightarrow\HH_2,\qquad (\tau_1,\tau_2)\longmapsto\big(\begin{smallmatrix}
    \tau_1 & 0\\
    0 & \tau_2
\end{smallmatrix}\big).
\end{equation}
Similar to the genus~$1$ case, we have~$\int_{\mathcal{F}_\Lambda} \KMtheta{2}\wedge\omega^{n-2}=\vol(\mathcal{F}_\Lambda)E^k_{2,\Lambda}$, where~$E^k_{2,\Lambda}$ is a vector-valued normalized Siegel Eisenstein series of genus~$2$; see Section~\ref{sec;vv Siegel mod}.
Thus, the modular form~\eqref{eq;doublthetaKM} is, up to the volume of~$\mathcal{F}_\Lambda$,  the pullback of~$E^k_{2,\Lambda}$ with respect to the embedding~\eqref{eq;stdembHH}.
A similar pullback is described by Garret~\cite{garrett} and Böcherer~\cite{boecherer} in their works on the doubling method; see~\cite[Section~2.9]{bruinierzuffetti} for a vector-valued generalization.
We describe the Fourier coefficients of the pullback~$E^k_{2,L}\big(\begin{smallmatrix}
    \tau_1 & 0\\
    0 & \tau_2
\end{smallmatrix}\big)$ in terms of elliptic Eisenstein and Poincaré series, and deduce a formula for~$\vol(H^2_{-m,\mu})$, which we use to compute~$\vol(P^2_{-m,\mu})$, see Proposition~\ref{prop:volprHeegner Siegel}.

 \subsection*{Acknowledgements} This paper benefited from helpful discussions and correspondence with the
following people who we gratefully acknowledge: Pietro Beri, Jan Bruinier, Phil Engel, Paul Kiefer, Radu Laza, Martin M\"oller, Scott Mullane, and Brandon Williams.

\section{Preliminaries} \label{sec2}

Let $\Lambda$ be an even lattice of signature~$(2,n)$ with bilinear form~$\langle\cdot{,}\cdot\rangle$ and associated quadratic form~$q(z)=\langle z,z\rangle/2$. Both ~$\langle\cdot{,}\cdot\rangle$ and~$q$ extend to $V=\Lambda\otimes \mathbb{R}$ and to $V_\mathbb{C}=\Lambda\otimes \mathbb{C}$.
The analytic open subset
\[
\left\{\left[z\right]\in\mathbb{P}(V_{\mathbb{C}})\,|\,q(z)=0\hbox{ and }\langle z,\overline{z}\rangle>0\right\}
\]
of the quadric in~$\mathbb{P}(V_{\mathbb{C}})$ defined by~$q$ has two connected components~$\mathcal{D}_{\Lambda}$ and~$\overline{\mathcal{D}}_{\Lambda}$, which are exchanged by complex conjugation.
Both components are projective models of the Hermitian symmetric domain of type IV arising from~$G={\rm{O}}(V)$.

Let~$\widetilde{\rm{O}}^+(\Lambda)$ be the subgroup of the discriminant kernel of~$G$ preserving~$\mathcal{D}_\Lambda$.
The \textit{orthogonal modular variety} arising from~$\Lambda$ is defined as
\[
\mathcal{F}_{\Lambda}=\mathcal{D}_\Lambda/\widetilde{\rm{O}}^+(\Lambda).
\]
    It is a quasi-projective variety~\cite{BB66} that for various choices of the lattice $\Lambda$ gives a partial compactification of a coarse moduli space of polarized varieties. For instance, when $\Lambda$ is given by the lattice
\[
\Lambda_{2d}=U^{\oplus 2}\oplus E_8(-1)^{\oplus 2}\oplus \mathbb{Z}\ell,\;\;\hbox{with}\;\;q(\ell)=-d,
\]
then the quotient $\mathcal{F}_{2d}=\mathcal{D}_{\Lambda_{2d}}/\widetilde{{\rm{O}}}^+(\Lambda_{2d})$ is the coarse moduli space of quasi-polarized K3 surfaces $(S,H)$.
Here quasi-polarized means that~$H$ is taken to be big and nef, but not necessarily ample. 

\subsection{Heegner divisors}\label{sec1.1:prelim} For fixed $v\in \Lambda^{\vee}\subset\Lambda_{\mathbb{Q}}$, set 
\[
D_v=v^\perp\cap \mathcal{D}_{\Lambda}=\left\{[z]\in \mathcal{D}_{\Lambda}\left|\langle z,v\rangle=0\right.\right\}.
\]
Note that if~$q(v)\geq 0$ and~$v\neq 0$, then~$D_v=\emptyset$.
Let~$\mu\in D(\Lambda)$ and let~$m\in \mathbb{Z}-q(\mu)$ be positive. Then the cycles 
\begin{equation}
\label{sec2:eq:Heegner}
\sum_{\substack{v\in \mu+\Lambda\\q(v)=-m}}D_v\qquad\hbox{and}\qquad\sum_{\substack{v\in \mu+\Lambda\\v\; \hbox{\tiny{primitive}}\\q(v)=-m}}D_v
\end{equation}
on $\mathcal{D}_\Lambda$ descend to $\mathbb{Q}$-Cartier divisors on $\mathcal{F}_\Lambda$ denoted $H^\Lambda_{-m,\mu}$ and $P^\Lambda_{-m,\mu}$ respectively. They are called a {\textit{Heegner divisor}} and {\textit{primitive Heegner divisor}} respectively. Whenever the lattice~$\Lambda$ with respect to which we construct Heegner divisors is clear from the context, we may drop the index~$\Lambda$ and simply write~$H_{-m,\mu}$.

In general, both $H_{-m,\mu}$ and $P_{-m,\mu}$ are neither reduced, nor irreducible. They relate as
\begin{equation}\label{eq: Heegner decomp}
H_{-m,\mu}=\sum_{\substack{r\in\ZZ_{>0} \\ r^2|m}}\sum_{\substack{\delta\in D(\Lambda)\\r\delta=\mu}}P_{-m/r^2,\delta}
\qquad\hbox{and}\qquad
P_{-m,\delta}=\sum_{\substack{r\in\ZZ_{>0} \\ r^2| m}}\mu(r)\sum_{\substack{\beta\in D(\Lambda)\\r\beta=\delta}}H_{-m/r^2,\beta},
\end{equation}
where $\mu(\cdot)$ is the M\"{o}bius function, see \cite[Lemma 4.2]{BM19}.
Here we say that~$r^2|m$ if there exists~$\beta\in D(\Lambda)$ with~$m/r^2\in \ZZ-q(\beta)$.
All the components of~$H_{-m,\mu}$ (and therefore of~$P_{-m,\mu}$) have multiplicity two if $\mu=-\mu$ in $D(\Lambda)$ and all have multiplicity one otherwise. Further, the line bundle $\mathcal{O}(-1)$ on $\mathcal{D}_{\Lambda}\subset \mathbb{P}\left(\Lambda_{\mathbb{C}}\right)$ admits a natural~$\widetilde{{\rm{O}}}^+(\Lambda)$-action and descends to a $\mathbb{Q}$-line bundle $\lambda$ called the {\textit{Hodge bundle}}. One defines $H_{0,0}=-\lambda$. 

When $\Lambda$ splits off two copies of $U$, then the pair $(m,\mu)$ determines the $\widetilde{{\rm{O}}}^+(\Lambda)$-orbit of primitive elements. This implies that $P_{-m,\mu}$ is irreducible in that case; see~\cite[Lemma~4.3]{BM19}.

\begin{example}\label{ex: NL Heegner relationship} In the K3 case $\mathcal{F}_{2d}$, these Heegner divisors correspond to the Noether--Lefschetz divisors $D_{h,a}$ defined in the introduction: One has (up to possibly a factor of~$1/2$) that $H_{-m,\mu}=D_{h,a}$ in ${\rm{Pic}}_{\mathbb{Q}}(\mathcal{F}_{2d})$ where
\[
m=\frac{a^2}{4d}-(h-1)\;\;\;\hbox{and}\;\;\;\mu=a\cdot\frac{\ell}{2d}\in D(\Lambda_{2d}),
\]
see \cite[Section~1 and Lemma~3]{MP13}.
\end{example}

Let~${\rm{Pic}}_{\mathbb{Q}}^{H}(\mathcal{F}_\Lambda)$ denote the subspace of the rational Picard group~$\Pic_\QQ(\mathcal{F}_\Lambda)$ generated by all Heegner divisors. Then, under the assumption that $\Lambda$ splits off two copies of the hyperbolic plane we have the equality \cite[Theorem 1.5]{BLMM17}, \cite[Remark 3.13, Corollary 3.18]{bruinierzuffetti}
\begin{equation}\label{eq: BLMM}
{\rm{Pic}}_{\mathbb{Q}}^H(\mathcal{F}_\Lambda)={\rm{Pic}}_{\mathbb{Q}}(\mathcal{F}_\Lambda).
\end{equation}

\subsection{Vector-valued Siegel modular forms}\label{sec;vv Siegel mod}

Let $\HH_\genus$ denote the Siegel upper half-space of genus~$\genus$, meaning the set of complex~$\genus\times\genus$-matrices with positive-definite imaginary part.
The symplectic group~$\Sp_{2\genus}(\RR)$ acts transitively on~$\HH_\genus$ as follows.
For any~$\gamma=\big(\begin{smallmatrix}
    a & b\\
    c & d
\end{smallmatrix}\big)\in\Sp_{2\genus}(\RR)$ written in blocks of~$\genus\times\genus$-matrices, the action of~$\gamma$ on~$\HH_\genus$ is defined as
\[
\gamma\cdot\tau = (a\tau+b)(c\tau + d)^{-1},\qquad \tau\in\HH_\genus.
\]

The {\textit{metaplectic group}} ${\rm{Mp}}_{2\genus}(\mathbb{R})$ is the double cover of $\mathrm{Sp}_{2\genus}(\mathbb{R})$ defined as the group of pairs $(A,\phi(\tau))$ where $A=\big(\begin{smallmatrix} a&b\\c&d\end{smallmatrix}\big)\in {\rm{Sp}}_{2\genus}(\mathbb{\RR})$, and $\phi(\tau)$ is a choice of a square root of the function $\det(c\tau+d)$ on the upper half space $\mathbb{H}_\genus$.
The product in ${\rm{Mp}}_{2\genus}(\mathbb{R})$ is given by
$(A_1,\phi_1(\tau))\cdot(A_2,\phi_2(\tau))=(A_1A_2, \phi_1(A_2\tau)\phi_2(\tau))$, and the action of~$\Mp_{2\genus}(\RR)$ on~$\HH_\genus$ is the one induced by~$\Sp_{2\genus}(\RR)$.

Let~$L$ be an even lattice of signature~$(m,n)$, and let~$\mathbb{C}\left[D(L)^\genus\right]$ be the group algebra generated by~$\genus$ copies of the discriminant group of $L$.
Its standard set of generators is~$\{\mathfrak{e}_{\discel}\mid \discel\in D(L)^\genus\}$.
Let 
  \[\rho_{L,\genus}:{\rm{Mp}}_{2\genus}(\mathbb{Z})\longrightarrow {\rm{GL}}\left(\mathbb{C}[D(L)^\genus]\right)\]
be the \emph{Weil representation} of ${\rm{Mp}}_{2\genus}(\mathbb{Z})$ on~$\mathbb{C}[D(L)^\genus]$; see e.g.~\cite[Section~2.1.3]{zh;phd} for a concrete description in terms of the standard generators of ${\rm{Mp}}_{2\genus}(\mathbb{Z})$ and for its realization as the restriction of the Schrödinger model of the Weil representation of~$\Mp_{2\genus}(\mathbb{A}_f)$ over the finite adeles.
If~$\genus=1$, then~$\rho_{L,\genus}$ is the same representation considered by Borcherds in~\cite[Section~4]{Bor98}.
 In this paper we are primarily interested in \emph{the dual}~$\weil{L,\genus}$ of this representation.

    Let~$k\in\frac{1}{2}\ZZ$.
    A \emph{Siegel modular form} of weight~$k$ and genus~$\genus$ with respect to the dual Weil representation~$\weil{L,\genus}$ is a holomorphic function~$f\colon \HH_\genus\to\CC[D_L^\genus]$ such that
    \[
    f(\gamma\cdot \tau)=\phi(\tau)^{2k}\weil{L,\genus}(\gamma) f(\tau)
    \qquad
    \text{for all~$\gamma=(M,\phi)\in\Mp_{2\genus}(\ZZ)$.}
    \]
    If~$\genus=1$, then we also require that~$f$ is holomorphic at~$\infty$.
    We denote the space of such modular forms by~$M^k_{\genus,L}$.

    To simplify the notation, we define~$e(t)\coloneqq\exp(2\pi i t)$ for every~$t\in \CC$.
    Let~$\halfint_\genus$ be the set of symmetric half-integral~$\genus\times\genus$-matrices.
These are matrices with coefficients in~$\frac{1}{2}\ZZ$ where the diagonal entries are integers.
If~$T\in\halfint_\genus$ is positive definite, resp.\ positive semidefinite, we write~$T>0$, resp.~$T\geq 0$.
    Every~$f\in M^k_{\genus,L}$ admits a Fourier expansion of the form
    \[
    f(\tau)=\sum_{\discel\in D_L^\genus}
    \sum_{\substack{T\in \halfint_\genus - q(\discel) \\ T\geq 0}}
    c_{T,\discel}(f) q^T \mathfrak{e}_\discel,
    \qquad
    \text{where } q^T=e(\trace T\tau).
    \]

    For~$\genus=1$, we denote by~$(M^k_{1,L})^\circ$ the space of \emph{almost cusp forms}, i.e.\ the subspace of modular forms~$f\in M^k_{1,L}$ such that if~$\mu\neq 0$, then~$c_{0,\mu}(f)=0$.

    Classical examples of modular forms are the Eisenstein series, which we now recall; see~\cite{bruinier-habilitation} and~\cite[Section~$2.4$]{bruinierzuffetti} for further information.
        Let~$\Gamma^g_\infty$ be the preimage under the metaplectic cover of the Siegel parabolic subgroup of~$\Sp_{2\genus}(\ZZ)$.
        The \emph{weight~$k$ Siegel Eisenstein series of genus~$\genus$ and type~$\weil{L,\genus}$}, arising from the zero of~$D(L)^\genus$, is defined as 
        \[
        E^k_{\genus,L}(\tau)=\sum_{\gamma\in\Gamma_\infty^\genus\backslash\Mp_{2\genus}(\ZZ)}
        \phi(\tau)^{-2k} \weil{L,\genus}(\gamma)^{-1} \mathfrak{e}_0,\qquad\tau\in\HH_\genus,
        \]
        where~$\mathfrak{e}_0$ above is the basis vector of~$\CC[D(L)^\genus]$ arising from the zero tuple in~$D(L)^\genus$.
        It is easy to see that if~$\mu\neq 0$ in~$D(L)^\genus$, then~$c_{0,\mu}(E^k_{\genus,L})=0$.
        Therefore~$E^k_{1,L}\in(M^k_{1,L})^\circ$.
        If the genus is clear from the context, we will drop it from the notation and simply write~$E_{k,L}$.

        Other examples of modular forms are the theta series attached to definite lattices.
        For the purposes of the present paper, it is enough to recall them under the assumption that the lattice~$L$ is \emph{negative definite}, i.e.~$m=0$.
        Under these assumptions, the Siegel theta series of genus~$g$ associated to~$L$ is the modular form~$\Theta_{L,\genus}\in M^{\rk L/2}_{\genus,L}$ defined as
        \begin{equation}\label{eq;defthetanegdef}
        \Theta_{L,\genus}(\tau)=
        \sum_{\discel\in D(L)^\genus} \sum_{\lambda\in\discel + L^\genus}
        e\big(
        -\trace q(\lambda)  \tau
        \big) \mathfrak{e}_\discel,\qquad \tau\in\HH_\genus.
        \end{equation}
        Let~$\theta(T,\discel)$ denote the Fourier coefficient of index~$(T,\discel)$ of~$\Theta_{L,\genus}$, with~$\discel\in D(L)^\genus$ and~$T\in \halfint_\genus -q (\discel)$ positive semidefinite.
        Then
        \[
        \theta(T,\discel)=\#\{\lambda\in \discel + L^\genus : -q(\lambda)=T\}.
        \]
       In particular~$\Theta_{L,1}\in (M^k_{1,L})^\circ$.

\subsection{Generating series and theta functions}\label{sec;genseries}
    From now on, we fix an even lattice~$\Lambda$ of signature~$(2,n)$, and a weight~$k\coloneqq 1+n/2$.
    Let $[H_{-m,\discel}^\Lambda]$ denote the \emph{cohomology class} of the Heegner divisor of index~$(m,\discel)$ in the de Rham cohomology group~$H^2(\mathcal{F}_\Lambda,\CC)$.
    The \emph{generating series of Heegner divisors} on the modular variety~$\mathcal{F}_\Lambda$ is
\begin{equation}\label{eq;genseries}
    \sum_{\discel\in D(\Lambda)}
    \sum_{\substack{m\in \ZZ - q(\discel)\\ m\ge 0}}
    [H_{-m,\discel}^\Lambda] q^m \mathfrak{e}_\discel,
    \qquad\tau\in\HH.
\end{equation}
Borcherds~\cite{borcherds-GKZ} proved that this generating series converges to an almost cusp form of weight~$k$ with respect to~$\weil{\Lambda,1}$ and with values in~$H^2(\mathcal{F}_\Lambda,\CC)$, i.e.\ it is to an element of~$(M^k_{1,\Lambda})^\circ\otimes H^2(\mathcal{F}_\Lambda,\CC)$.
This result has been generalized by Kudla and Millson in~\cite{kudlamillson-intersection}, where they proved that the generating series of the so-called \emph{special cycles} of codimension~$\genus$ is a Siegel modular form of weight~$k$ with respect to~$\weil{\Lambda,\genus}$. The special cycles of codimension~$1$ are the Heegner divisors introduced above.

The convergence and modularity of the generating series of special cycles is proved in~\cite{kudlamillson-intersection} by realizing such series as the cohomology class of a theta function, nowadays known as the \emph{Kudla--Millson theta function}.
For the purposes of the present paper, we quickly recall how to construct it; see~\cite{kudlamillson-harmonicI, kudlamillson-harmonicII} for details.

Let~$V=\Lambda\otimes\RR$.
Fix a point~$z_0\in \mathcal{D}_\Lambda$ once and for all.
Recall the stabilizer~$\Kmax$ of $z_0$ is a maximal compact subgroup of~$G\coloneqq\bigO(V)$ and that we may realize the symmetric domain~$\mathcal{D}_\Lambda$ as the quotient~$G/\Kmax$.
The \emph{Kudla--Millson Schwartz function}~$\KMschwartzg{\genus}$ is a~$G$-invariant Schwartz function on~$V^\genus$ with values in the space~$\mathcal{A}^{2\genus}(\mathcal{D}_\Lambda)$ of differential~$2\genus$-forms on~$\mathcal{D}_\Lambda$, in short $\KMschwartzg{\genus}\in [\mathcal{S}(V^\genus)\otimes \mathcal{A}^{2\genus}(\mathcal{D}_\Lambda)]^G$.

Since the action of $G$ on $\mathcal{D}_\Lambda$ is transitive, in order to define~$\KMschwartzg{\genus}$ it is enough to construct it as a~$\Kmax$-invariant Schwartz function at the base point~$z_0$.
Kudla and Millson constructed the~$\Kmax$-invariant value of~$\KMschwartzg{\genus}$ at~$z_0$ as an element
    \begin{equation}\label{eq;KMschwartzbasept}
    \KMschwartzg{\genus}(z_0)\in\Big[
    \mathcal{S}(V^\genus)\otimes{\bigwedge}^{2\genus} T_{z_0}^*\mathcal{D}_\Lambda
    \Big]^{\Kmax}
    \end{equation}
    by applying a special differential operator to a Gaussian on~$V^\genus$; see~\cite[Section~$3$]{kudlamillson-harmonicI}.
    
The Kudla--Millson theta function~$\KMtheta{\genus}^\Lambda$ associated to~$\Lambda$ is defined as
\begin{equation}\label{eq: KM theta}
\KMtheta{\genus}^\Lambda(\tau, z)
=
\det y^{-k/2}
\sum_{\discel\in D(\Lambda)^\genus}
\sum_{\lambda\in \discel + \Lambda^\genus}
\big(
\omega_\infty(g_\tau)\KMschwartzg{\genus}
\big)(\lambda,z)\mathfrak{e}_\discel,
\quad\text{$\tau\in\HH_\genus$, $z\in\mathcal{D}_\Lambda$,}
\end{equation}
where $\omega_\infty$ is the dual of the (Schrödinger model of the) Weil representation of~$\Mp_{2\genus}(\RR)\times G$ acting on the space~$\mathcal{S}(V^\genus)$ of Scwhartz functions on~$V^\genus$, and~$g_\tau=\big(\begin{smallmatrix}
    1 & x \\
    0 & 1
    \end{smallmatrix}\big)
    \big(\begin{smallmatrix}
    y^{1/2} & 0 \\
    0 & (y^{-1/2})^t
    \end{smallmatrix}\big)$ is the standard element of~$\Sp_{2\genus}(\RR)$ mapping~$iI_\genus$ to~$\tau=x+iy$.
    If it is clear from the context, we will sometime drop~$\Lambda$ and simply write~$\KMtheta{\genus}$.
   
The theta function~$\KMtheta{\genus}$ has several remarkable properties.
    It is invariant with respect to pull-backs by isometries in~$\widetilde{\rm{O}}^+(\Lambda)$, hence it descends to a function on~$\mathcal{F}_\Lambda$.
    Furthermore, it transforms as a (non-holomorphic) Siegel modular form of weight~$k$ with respect to~$\weil{\Lambda,\genus}$ with values in~$\CC[D(\Lambda)^\genus]\otimes\mathcal{A}^{2\genus}(\mathcal{F}_\Lambda)$.
    Since~$\KMschwartz$ is a \emph{closed} differential form by~\cite[Theorem~3.1]{kudlamillson-harmonicI}, then the~$2\genus$-form~$\KMtheta{\genus}$ is closed on~$\mathcal{F}_\Lambda$, hence we may consider its de Rham cohomology class, which we denote by~$[\KMtheta{\genus}]$.
    It is proved in~\cite{kudlamillson-intersection} that~$[\KMtheta{\genus}]$ is a \emph{holomorphic} Siegel modular form with respect to the symplectic variable~$\tau$, namely~$[\KMtheta{\genus}]\in M^k_{\genus,\Lambda}\otimes H^{2\genus}(\mathcal{F}_\Lambda,\CC)$, and its Fourier expansion is the generating series of codimension~$\genus$ special cycles on~$\mathcal{F}_\Lambda$.
    If~$\genus=1$, then~$[\KMtheta{1}]$ is equal to the generating series~\eqref{eq;genseries}.

\subsection{Pulling back Heegner divisors along lattice embeddings}
Suppose that $L$ and $\Lambda$ are two even lattices of signatures $(2,n')$ and $(2,n)$ respectively such that we have a lattice embedding $L\subset \Lambda$.
This embedding induces a map of modular varieties
\begin{equation}\label{eq: finite map}
\varphi\colon \mathcal{F}_{L}\rightarrow \mathcal{F}_{\Lambda}
\end{equation}
as well as pullback maps in cohomology
\[
\varphi^*\colon H^{2\genus}(\mathcal{F}_\Lambda,\CC)\longrightarrow H^{2\genus}(\mathcal{F}_L,\CC).
\]

As illustrated in~\cite[Section~$9$]{kudla-algcycles}, the pullback of Heegner divisors on~$\mathcal{F}_\Lambda$ under~$\varphi^*$, with~$\genus=1$, can be explicitly written in terms of the Heegner divisors on~$\mathcal{F}_L$.
The cited result is stated in an adelic setting.
We provide here a version of it in terms of lattice cosets, suitable for the purposes of the present paper.
For any sublattice~$\Omega\subset\Lambda$ of finite index, let~$\trg{\genus}{\Lambda}{\Omega}\colon M^k_{\genus,\Omega}\to M^k_{\genus,\Lambda}$ be the map
such that the~$\overline{\discel}$-component of~$\trg{\genus}{\Lambda}{\Omega}(f)$ is
\[
\sum_{\alpha\in(\Lambda/\Omega)^\genus} f_{\alpha + \discel}
\]
for all~$f\in M^k_{\genus,\Omega}$ and~$\overline{\discel}\in D(\Lambda)^\genus$, where~$\discel$ is any fixed preimage of~$\overline{\discel}$ under the natural quotient map~$(\Lambda^\vee/\Omega)^\genus \to (\Lambda^\vee/\Lambda)^\genus$. See~\cite[Section 2.3]{bruinierzuffetti} for further details on~$\trg{\genus}{\Lambda}{\Omega}$.

\begin{proposition}[Kudla]\label{prop:pulbackheegner}
 Let $\Lambda$ be an even lattice of signature $(2,n)$ and let $L$ be a sublattice of signature $(2,n')$ with $n'\le n$. Writing $K=L^\perp$ and $\Omega=L\oplus K$, we have
\begin{equation}\label{eq:main eqkudla}
    \varphi^*\big(
    [\KMtheta{\genus}^\Lambda]
    \big)=
    \trg{\genus}{\Lambda}{\Omega}\big(
    [\KMtheta{\genus}^L]\otimes\Theta_{K,g}
    \big).\end{equation}
    In particular, for any~$\overline{\discel}\in D(\Lambda)$ and~$m\in \ZZ - q(\overline{\discel})$ non-negative we have that
    \begin{equation}\label{eq;2main eqkudla}
        \varphi^*([H_{-m,\overline{\discel}}^\Lambda])
        =
        \sum_{\alpha\in \Lambda/\Omega}
        \sum_{\substack{t\in \ZZ - q((\alpha + \discel)_L) \\ t\geq 0}}
        \theta(m-t,(\alpha + \discel)_K)
        [H_{-t,(\alpha + \discel)_L}^L]
    \end{equation}
    where~$\discel\in \Lambda^\vee/\Omega$ is a fixed preimage of~$\overline{\discel}$ under~$\Lambda^\vee/\Omega \to \Lambda^\vee/\Lambda$.
\end{proposition}
\begin{proof}
    We may factor~$\varphi$ in \eqref{eq: finite map} as the composition of analogous maps
\[\mathcal{F}_L\xrightarrow{\varphi_L}\mathcal{F}_\Omega\xrightarrow{\varphi_\Omega}\mathcal{F}_\Lambda,\]
where~$\varphi_\Omega$ is a finite cover.
As illustrated in~\cite[Remark~3.10]{bruinierzuffetti} we have that $\varphi^*\big(
        [\KMtheta{\genus}^\Lambda]
        \big)
        =
        \trg{\genus}{\Lambda}{\Omega}\big(
        \varphi_L^*([\KMtheta{\genus}^\Omega])
        \big)$.
    To prove~\eqref{eq:main eqkudla}, it is then enough to show that $\varphi_L^*\big(
[\KMtheta{\genus}^\Omega]
\big)
=
[\KMtheta{\genus}^L]\otimes\Theta_{K,\genus}$.
This follows from the equality of differential forms
    \begin{equation}\label{eq;wwwtprestrtheta}
        \KMtheta{\genus}^\Omega|_{\mathcal{D}_L}
        =
        \KMtheta{\genus}^L\otimes\Theta_{K,\genus}
    \end{equation}
    on the Hermitian symmetric domain~$\mathcal{D}_L$, which we now prove.
    Here~$\KMtheta{\genus}^\Omega|_{\mathcal{D}_L}$ denotes the restriction of the differential form~$\KMtheta{\genus}^\Omega$ to~$\mathcal{D}_L$.

     Let~$V_\Omega=\Omega\otimes\RR$. Define~$V_L$ and~$V_K$ similarly.
     Recall that~$\KMschwartzg{\genus}^\Omega$ is constructed as the pullback under the action of~$G=\bigO(V_\Omega)$ on~$\mathcal{D}_\Omega$ of the value~$\KMschwartzg{\genus}(z_0)$ at a base point.
    Without loss of generality, we may choose~$z_0$ to lie on~$\mathcal{D}_L$.
    From the very definition of the Kudla--Millson Schwartz function at~$z_0$, see e.g.~\cite[Theorem~2.1 (ii)]{funke}, one can show that
    \[
    \KMschwartzg{\genus}^\Omega|_{\mathcal{D}_L}(v,z)
    =
    \KMschwartzg{\genus}^L(v_1,z)\varphi_{V_K,\genus}(v_2),\qquad \text{$v=(v_1,v_2)\in V_L^\genus\oplus V_K^\genus$, $z\in\mathcal{D}_L$,}
    \]
    where~$\varphi_{V_K,\genus}$ is the standard Gaussian on~$V_K^\genus$.
    We may then rewrite $\KMtheta{\genus}^\Omega|_{\mathcal{D}_L}(\tau,z)$ as
    \begin{equation}\label{eq;decrestrict}
    \det y^{-\rk\Omega/4}
         \sum_{\substack{\discel\in D(L)^\genus \\ \lambda_1\in \discel + L^\genus}}
        \sum_{\substack{\kappa\in D(K)^\genus \\ \lambda_2\in \kappa + K^\genus}}
        \big(
        \omega_\infty(g_\tau)\KMschwartzg{\genus}^L
        \big)(\lambda_1,z)
        \cdot
        \big(\omega_\infty(g_\tau)\varphi_{V_K,\genus}\big)(\lambda_2)\mathfrak{e}_\discel\otimes\mathfrak{e}_\kappa
    \end{equation}
    for all~$z\in\mathcal{D}_L$.
    From the above formula, one can easily check that the summation over~$(\kappa,\lambda_2)$ equals~$\det y^{\rk K/4}
        \Theta_{K,\genus}(\tau)$.
    This, replaced in~\eqref{eq;decrestrict}, yields~\eqref{eq;wwwtprestrtheta}.
\end{proof}

Let~$\psi_\Lambda\colon \big((M^{\rk \Lambda/2}_{1,\Lambda})^\circ\big)^\vee\to H^{2}(\mathcal{F}_\Lambda,\CC)$ be the function defined by~$c_{m,\discel}\mapsto [H_{-m,\discel}]$ for all~$\discel\in D(\Lambda)$ and~$m\in\ZZ-q(\discel)$ with~$m>0$.
With the same notation as Proposition~\ref{prop:pulbackheegner}, we deduce that the following diagram commutes.
\begin{equation}\label{eq;gencommdiagrcoho}
    \begin{tikzcd}[column sep=15ex]
    \big((M^{\rk\Lambda/2}_{1,\Lambda})^\circ\big)^\vee \arrow[r, "\quad(\Theta_{K,1})^\vee\circ\,\trg{1}{\Lambda}{\Omega}^\vee\quad"] \arrow[d, "\psi_\Lambda"]
    &
    \big((M^{\rk L/2}_{1,L})^\circ\big)^\vee \arrow[d, "\psi_L"] \\
    H^{2}(\mathcal{F}_\Lambda,\CC) \arrow[r, "\varphi^*"]
    &
    H^{2}(\mathcal{F}_L,\CC)
    \end{tikzcd}
    \end{equation}
Here~$\trg{1}{\Lambda}{\Omega}^\vee\colon \big((M^{\rk \Lambda/2}_{1,\Lambda})^\circ\big)^\vee\to \big((M^{\rk \Lambda/2}_{1,\Omega})^\circ\big)^\vee$ denotes the dual of the trace map, and $\Theta_{K,1}^\vee\colon \big((M^{\rk \Lambda/2}_{1,\Omega})^\circ\big)^\vee\to \big((M^{\rk L/2}_{1,L})^\circ\big)^\vee$ is the dual of~$(M^{\rk L/2}_{1,L})^\circ\to (M^{\rk\Lambda/2}_{1,\Omega})^\circ$, $f\mapsto f\otimes\Theta_{K,1}$.

Let~$(M^{k}_{1,\Lambda})^\circ(\QQ)$ denote the space of weight~$k$ elliptic modular forms with respect to~$\weil{\Lambda}$ and with \emph{rational} Fourier coefficients.
 Under the additional assumption that $n,n'\ge 3$, one has that $H^1(\mathcal{F}_\Lambda, \mathbb{C})=H^1(\mathcal{F}_L, \mathbb{C})=0$, see \cite{Kon88} and \cite[Remark 4.1]{BLMM17}.
 It follows that in this case the commutativity of \eqref{eq;gencommdiagrcoho} may be upgraded to
 \begin{equation}\label{eq;gencommdiagrpic}
   \begin{tikzcd}[column sep=15ex]
    \big((M^{\rk\Lambda/2}_{1,\Lambda})^\circ(\QQ)\big)^\vee \arrow[r, "\quad(\Theta_{K,1})^\vee\circ\,\trg{1}{\Lambda}{\Omega}^\vee\quad"] \arrow[d, "\psi_\Lambda"]
    &
    \big((M^{\rk L/2}_{1,L})^\circ(\QQ)\big)^\vee \arrow[d, "\psi_L"] \\
    \Pic_\QQ(\mathcal{F}_\Lambda) \arrow[r, "\varphi^*"]
    &
    \Pic_\QQ(\mathcal{F}_L)
    \end{tikzcd}
    \end{equation}
    where the maps~$\psi_\Lambda$ and~$\psi_L$ are constructed as their homonym in cohomology, replacing cohomology classes of Heegner divisors with rational classes.

    \section{Extremal rays of $\overline{{\rm Eff}}(\mathcal{F}_\Lambda)$}

    An effective divisor $D$ on a variety $X$ is  \textit{extremal} if it does not admit a non-trivial effective decomposition, that is, for any linear combination $D = a_1D_1 + a_2D_2$ in ${\rm{Pic}}_{\mathbb{R}}(X)$ with $a_i > 0$ and $D_i$ pseudo-effective, the divisors $D$ and $D_i$ are proportional. In this case, we say that the divisor $D$ spans an extremal ray of the pseudo-effective cone $\overline{{\rm Eff}}(X)\subset {\rm{Pic}}_{\mathbb{R}}(X)$.
    In order to prove that a divisor is extremal we will use the following adaptation of the numerical criterion in \cite[Lemma 4.1]{CC14} to show extremality in~$\overline{{\rm Eff}}(Y)\subset {\rm{NS}}_{\mathbb{R}}(Y)$ for $Y$ normal projective and $\mathbb{Q}$-factorial.
    Recall that given an irreducible divisor $D\subset Y$ and a curve $C\subset Y$ the class $\left[C\right]\in{\rm{N}}_1(Y)$ is said to be a {\it{moving curve for}} $D$ if $C$ deforms over an algebraic base such that the total family covers a dense subset of $D$. 
  
    \begin{lemma}\label{lem: cc lem}
       Let $X^{BB}$ be a normal projective variety, $Z\subset X^{BB}$ closed of codimension ${\rm{codim}}(Z)\geq 2$, and assume the complement $X=X^{BB}\setminus Z$ is $\mathbb{Q}$-factorial with trivial Albanese variety. Let $X^c\longrightarrow X^{BB}$ be a $\mathbb{Q}$-factorial compactification of $X$ over $X^{BB}$. Let $D$ be an irreducible effective divisor on $X$ and $D^c$ its closure in $X^c$. If $\left[C\right]$ is a moving curve for $D^c$ such that  $C\cap \partial X^c=\emptyset$ and $D^c\cdot \left[C\right]<0$, then: 
       \begin{enumerate}
           \item The divisor $D$ is extremal in $\overline{{\rm Eff}}(X)$.
           \item The closure $\overline{D}$ is extremal in $\overline{{\rm Eff}}(\overline{X})$, where $\overline{X}\longrightarrow X^{BB}$ is any $\mathbb{Q}$-factorial compactification of $X$ over $X^{BB}$. 
       \end{enumerate}
    \end{lemma}
\begin{remark}Note that for any even lattice $\Lambda$ of signature $(2,n)$ with $n\geq3$, the orthogonal modular variety $\mathcal{F}_\Lambda$ has trivial Albanese variety, see \cite{Kon88} and \cite[Remark~4.1]{BLMM17}. Therefore $\mathcal{F}_\Lambda\subset\mathcal{F}_\Lambda^{BB}$ satisfies the hypothesis of Lemma \ref{lem: cc lem}.
\end{remark}
\begin{proof}[Proof of Lemma~\ref{lem: cc lem}]
 Let $Y\subset X^c$ be the boundary. Begin by noting that since $X$ has trivial Albanese and the image of $Y$ under $X^c\longrightarrow X^{BB}$ has codimension at least $2$,  by \cite[Section~2]{Kon88} (see also \cite[Section~1]{Mum67}) we have 
\[
H_1(X^c,\mathbb{Q})=H_1(X,\mathbb{Q})=H_1({\rm{Alb}}(X),\mathbb{Q})=0.
\]
In particular the Albanese variety of $X^c$ is trivial and ${\rm{NS}}_{\mathbb{R}}(X^c)={\rm{Pic}}_{\mathbb{R}}(X^c)$. 

Now assume $D=a_1D_1+a_2D_2$ in ${\rm{Pic}}_{\mathbb{R}}(X)$ with $D_i$ effective and $a_i>0$. Note that since $[C]$ is a moving curve for $D^c$ and $C\cdot D^c<0$, the divisor $D^c$ is extremal in $\overline{\rm{Eff}}(X^c)$, see \cite[Lemma 4.1]{CC14}. Consider the localization sequence 
\begin{equation}
\label{sec4:eq:loc_seq}
{\rm{CH}}_{n-1}(Y)_{\mathbb{Q}}\overset{i_*}{\longrightarrow} {\rm{CH}}_{n-1}(X^c)_{\mathbb{Q}}\overset{j^*}{\longrightarrow}{\rm{CH}}_{n-1}(X)_{\mathbb{Q}}\longrightarrow 0,
\end{equation}
where $i:Y\hookrightarrow X^c$, $j:X\hookrightarrow X^c$ are the natural inclusions and $n=\dim X$.

If ${\rm{codim}}(Y)\geq 2$, then ${\rm{CH}}_{n-1}(Y)=0$ and $j^*$ is an isomorphism.
Hence $D^c=a_1D_1^c+a_2D_2^c$ in ${\rm{CH}}_{n-1}(X^c)_{\mathbb{Q}}={\rm{Pic}}_{\mathbb{Q}}(X^c)$.
In particular, since $D^c$ is extremal in~$\overline{\rm{Eff}}(X^c)$, we have that $D^c, D_1^c, D_2^c$ are all proportional in ${\rm{Pic}}_{\mathbb{R}}(X^c)$.
Restricting via the isomorphism~$j^*$ 
yields that $D, D_1, D_2$ are proportional in ${\rm{CH}}_{n-1}(X)_{\mathbb{Q}}={\rm{Pic}}_{\mathbb{Q}}(X)$. Therefore,~$D$ is extremal in~$\overline{\rm{Eff}}(X)$. 

If ${\rm{codim}}(Y)=1$, by \eqref{sec4:eq:loc_seq} (after extending coefficients) one has
\[
D^c=a_1D_1^c+a_2D_2^c-b\delta\;\;\hbox{ in }{\rm{Pic}}_{\mathbb{R}}(X^c)={\rm{NS}}_{\mathbb{R}}(X^c) 
\]
where $\delta$ is a divisor supported on $Y$.
Since $D^c$ is extremal in $\overline{\rm{Eff}}(X^c)$, we know~$b\geq 0$ and~$D^c+b\delta$ is effective.
After scaling $D^c+b\delta$ by an appropriate constant, we may assume that~$(D^c+b\delta)$ is Cartier.
Since $C\cap Y=\emptyset$, have $C\cdot (D^c+b\delta)=C\cdot D^c<0$. Because~$C$ covers a dense subset of $D^c$, for sufficiently large $m>0$ the divisor $D^c$ is in the base locus of the linear system $\left|m(D^c+b\delta)\right|$. More concretely
\begin{equation}
\label{sec3:eq:rigid}
\left|m(D^c+b\delta)\right|=\left|m(D^c+b\delta)-mD^c\right|+mD^c,
\end{equation}
cf.~\cite[Lemma 6.4]{BM24}. But $\delta$ is contracted via $X^c\longrightarrow X^{BB}$, in particular $\left|mb\delta\right|$ and therefore \eqref{sec3:eq:rigid} consists of a single rigid divisor, yet $m\left(a_1D_1^c+a_2D_2^c\right)\in \left|m(D^c+b\delta)\right|$. Since both $D_1^c$ and $D_2^c$ are not supported on $\delta$, this can only happen if $b=0$ and both $D_1^c, D_2^c$ are proportional to $D^c$. After restricting to~$X$ via~$j^*$, we have that~$D$, $D_1$, $D_2$ are all proportional. This finishes the argument that~$D$ is extremal in~$\overline{\rm{Eff}}(X)$.

For the second statement, note that the same argument as before shows that ${\rm{Pic}}_{\mathbb{R}}(\overline{X})$ equals~${\rm{NS}}_{\mathbb{R}}(\overline{X})$. Then extremality follows from the fact that $C$ covers a Zariski dense set in $\overline{D}$ and $C\cdot\overline{D}<0$.
\end{proof}

A {\textit{primitive representative}} for the pair $(-m,\mu)$ is a primitive element $\rho\in \Lambda$ such that $\mu=\frac{\rho}{{\rm{div}}_{\Lambda}(\rho)}+\Lambda\in D(\Lambda)$ and $-m=\frac{q(\rho)}{{\rm{div}}_{\Lambda}(\rho)^2}$, where the \textit{divisibility} ${\rm{div}}_{\Lambda}(\rho)$ of $\rho$ in $\Lambda$ is the positive generator of the ideal $\langle \rho, \Lambda\rangle\subset \ZZ$, or equivalently, the order $d_\mu$ of $\mu$.
If~$\Lambda$ splits off two copies of $U$, the primitive representative is unique up to the action of~$\widetilde{\rm{O}}^+(\Lambda)$.

\begin{corollary}
\label{cor: degree condition}
Let $\Lambda$ be an even lattice of signature $(2,n)$ with $n\ge 4$, and let~$\rho$ be a primitive representative of~$(-m,\mu)$. Let  $K=\mathbb{Z}\rho$, $L=\rho^{\perp \Lambda}$, and $\varphi\colon \mathcal{F}_L\rightarrow \mathcal{F}_\Lambda$ the map on orthogonal modular varieties induced by the inclusion $L\subset \Lambda$.
Denote by~$P$ the image of~$\varphi$.
If ${\rm \deg}(\varphi^*P)<0$, then 
\begin{enumerate}
    \item $P$ generates an extremal ray in $\overline{{\rm{Eff}}}(\mathcal{F}_\Lambda)$. 
    \item The closure $\overline{P}$ is an extremal ray in $\overline{{\rm{Eff}}}(\overline{\mathcal{F}}_\Lambda)$ for any normal projective $\mathbb{Q}$-factorial variety $\overline{\mathcal{F}}_\Lambda$ compactifying $\mathcal{F}_{\Lambda}$ over $\mathcal{F}_{\Lambda}^{BB}$.
\end{enumerate}   
\end{corollary}
In the above ${\rm{deg}}(D)$ stands for the degree of the closure $\overline{D}$ in the Baily--Borel model~$\mathcal{F}_L^{BB}\subset\mathbb{P}^N$. Equivalently, this is the degree of $\overline{D}$ in a toroidal compactification~$\mathcal{F}_L^{\rm tor}$ with respect to the nef class given by the pullback to~$\mathcal{F}_L^{\rm tor}$ of an appropriate multiple of~$\lambda_L$.

\begin{proof}
Consider the map~$\varphi\colon \mathcal{F}_L\rightarrow \mathcal{F}_{\Lambda}$
induced by the embedding $L\subset \Lambda$. Recall that 
\[
\mathcal{F}_{L}^{BB}={\rm{Proj}}\Big(\bigoplus_{k}H^0\big(\mathcal{F}_L,\lambda_L^{\otimes k}\big)\Big)\subset \mathbb{P}^N,
\]
where $H^0\left(\mathcal{F}_L,\lambda_L^{\otimes k}\right)$ can be seen as the space of scalar-valued orthogonal modular forms on the affine cone $\mathcal{D}_L^\bullet=\{z\in L\otimes \mathbb{C}\mid \mathbb{C}^*z\in \mathcal{D}_L\}$ of weight $k$ and with trivial character.
In particular, every such modular form on $\mathcal{D}_\Lambda^\bullet$ restricts to a modular form on $\mathcal{D}_L^\bullet$ of the same weight, giving us a map of graded algebras
\[
\bigoplus_{k}H^0\left(\mathcal{F}_\Lambda,\lambda_\Lambda^{\otimes k}\right)\longrightarrow \bigoplus_{k}H^0\left(\mathcal{F}_L,\lambda_L^{\otimes k}\right).
\]
The map $\varphi$ then extends to a map on Baily--Borel models $\varphi^{BB}:\mathcal{F}_L^{BB}\longrightarrow\mathcal{F}_\Lambda^{BB}$.
 Let~$\mathcal{F}_L^{\rm tor}$ be a toroidal compactification of $\mathcal{F}_L$, see \cite{AMRT10}, and consider the rational map~${\varphi^{\rm tor}: \mathcal{F}_L^{\rm tor}\dashrightarrow \mathcal{F}_\Lambda^{\rm tor}}$ extending $\varphi$ to the type II boundary. Let $\widetilde{\varphi}:\widetilde{\mathcal{F}_L}\longrightarrow\mathcal{F}_\Lambda^{\rm tor}$ be a resolution of $\varphi^{\rm tor}$. We thus have a diagram

\begin{equation}
\label{sec3:diag:alpha-beta}
\begin{tikzcd}
\widetilde{\mathcal{F}_L}\arrow[dd, bend right, "\widetilde{f}_L"']\arrow[r, "\widetilde{\varphi}"]\arrow[d]&\mathcal{F}_\Lambda^{\rm tor}\arrow[d, equal]\arrow[dd, bend left, "f_\Lambda"]\\
\mathcal{F}_L^{\rm tor}\arrow[r, dashed, "\varphi^{{\rm tor}}"]\arrow{d} &\mathcal{F}_\Lambda^{\rm tor}\arrow[d]\\
\mathcal{F}_L^{BB}\arrow[r, "\varphi^{BB}"]&\mathcal{F}_\Lambda^{BB}.
\end{tikzcd}
\end{equation}

Note that since the boundaries of both $\mathcal{F}_L^{BB}$ and $\mathcal{F}_{\Lambda}^{BB}$ are one dimensional and $\dim \mathcal{F}_L=\dim \mathcal{F}_\Lambda-1=n-1\geq 3$, both $\widetilde{f}_L$ and $f_\Lambda$ satisfy the hypotheses of the map~$X^c\rightarrow X^{BB}$ of Lemma~\ref{lem: cc lem}.

For some positive constant $r>0$, the class $r\lambda_L\subset \mathcal{F}_L$ corresponds to the hyperplane class coming from $\mathcal{F}_L^{BB}\subset\mathbb{P}^N$.
In particular it is Cartier and since $\dim \mathcal{F}_L=n-1\geq 3$ and the boundary $\mathcal{F}_L^{BB}\backslash \mathcal{F}_L$ is one-dimensional, we can find a representative $C$ of the curve class $[C]=\big(\widetilde{f}_L^*r\lambda_L\big)^{n-2}\in {\rm{N}}_1\big(\widetilde{\mathcal{F}_{L}}\big)$ which does not meet the boundary of $\widetilde{\mathcal{F}_{L}}$.
Note that $\left[C\right]$ is a movable curve on $\widetilde{\mathcal{F}_L}$ and since~$\varphi(\mathcal{F}_L)=P$, the curve class $\widetilde{\varphi}_*\left[C\right]\in {\rm{N}}_1\left(\mathcal{F}_{\Lambda}^{\rm tor}\right)$ is a moving curve for the divisor~$\overline{P}\subset \mathcal{F}_\Lambda^{\rm tor}$. 

Now we know
\[
\left[C\right]\cdot \overline{\varphi^*P}={\rm{deg}}(\varphi^*P)<0.
\]
Moreover the difference between $\widetilde{\varphi}^*\overline{P}$ and the closure in $\widetilde{\mathcal{F}_L}$ of $\varphi^*P$ is supported on the boundary of $\widetilde{\mathcal{F}_L}$. Thus since $C$ does not intersect the boundary, and~${[C]=\big(\widetilde{f}_L^*r\lambda_{L}\big)^{n-2}}$ we have
\[[C]\cdot\widetilde{\varphi}^*\overline{P}<0.\]

Then by the projection formula (e.g. \cite[Proposition 2.7]{Voi14} or \cite[Section 8.1]{Ful98}) we have
\[
\widetilde{\varphi}_*[C]\cdot \overline{P}=\widetilde{\varphi}_*\big([C]\cdot\widetilde{\varphi}^*\overline{P}\big)=[C]\cdot\widetilde{\varphi}^*\overline{P}<0.
\]
The result then follows from Lemma \ref{lem: cc lem}.
\end{proof}

We now detail two methods to establish that ${\rm \deg}(\varphi^*P)<0$ and thereby deduce the conclusion of Corollary \ref{cor: degree condition}. The first describes ${\rm \deg}(\varphi^*P)$ in terms of coefficients of the Eisenstein series $E_{\frac{n+1}{2},L}$ for $L$ and gives bounds on these coefficients guaranteeing that ${\rm \deg}(\varphi^*P)<0$. The second method works under the additional assumption that~$\Lambda$ splits off two hyperbolic planes, so that~$P=P_{-m,\mu}^{\Lambda}$. It describes ${\rm \deg}(\varphi^*P)$ in terms of the volume~${\rm vol}(P_{-m, \mu}^2)$ on $\mathcal{F}_\Lambda$ and computes this volume in terms of pullbacks of Siegel Eisenstein series.

\subsection{Extremality via Eistenstein series on a sublattice}
Let $\mu\in D(\Lambda)$ and $m\in \mathbb{Z}-q(\mu)$ non-negative. A primitive Heegner divisor $P_{-m, \mu}$ on an orthogonal modular variety $\mathcal{F}_\Lambda$ is nonempty if and only if the pair $(-m,\mu)$ has a primitive representative.
Recall that the latter is a primitive element $\rho\in \Lambda$ such that $\mu=\frac{\rho}{{\rm{div}}_{\Lambda}(\rho)}+\Lambda\in D(\Lambda)$ and~$-m=\frac{q(\rho)}{{\rm{div}}_{\Lambda}(\rho)^2}$.

\begin{theorem}
\label{thm: extremality2}
Let $\Lambda$ be an even lattice of signature $(2,n)$ with $n\ge 4$ and~$\rho$ a primitive representative of~$(-m,\mu)$.
Let~$K=\mathbb{Z}\rho$ and $L=\rho^{\perp \Lambda}$ and denote by~$P$ the divisor arising as the image of~$\varphi\colon\mathcal{F}_{L}\to\mathcal{F}_\Lambda$. Then $P$ is extremal in $\overline{{\rm{Eff}}}(\mathcal{F}_\Lambda)$ if
\begin{equation}
\label{sec3:eq:bound}
\begin{split}
&-\sum_{\substack{r\in\ZZ_{>0} \\ r^2| m}}\mu(r)
\sum_{\substack{\overline\beta\in D(\Lambda)\\r\overline\beta=\mu}}
\sum_{\substack{\alpha\in\Lambda/L\oplus K\\t\in\ZZ-q((\alpha+\beta)_L)\\ t>0}}
        \theta\left(\frac{m}{r^2}-t,(\alpha+\beta)_K\right)c_{t,(\alpha+\beta)_L}\left(E_{\frac{n+1}{2},L}\right)\\
        &\quad
        <\sum_{\substack{r\in\ZZ_{>0} \\ r^2| m}}\mu(r)\sum_{\substack{\overline\beta\in D(\Lambda)\\r\overline\beta=\mu}}
        \sum_{\substack{\alpha\in\Lambda/L\oplus K\\ (\alpha+\beta)_L = 0}}\theta\left(\frac{m}{r^2},(\alpha+\beta)_K\right),
\end{split}
\end{equation}
where~$\beta$ is a fixed preimage of~$\overline\beta$ under~$\Lambda^\vee/L\oplus K\to\Lambda^\vee/\Lambda$ and $\mu(\cdot)$ is the M\"{o}bius function. Further, its closure $\overline{P}$ is extremal in $\overline{{\rm{Eff}}}\left(\overline{\mathcal{F}}_\Lambda\right)$ for any normal projective $\mathbb{Q}$-factorial variety compactifying $\mathcal{F}_{\Lambda}$ over $\mathcal{F}_{\Lambda}^{BB}$.

\end{theorem}

\begin{proof}
By Corollary~\ref{cor: degree condition} it is enough to show that~${\rm deg}(\varphi^*P_{-m,\mu}^\Lambda)<0$, since~$\deg(\varphi^*P)\leq \deg(\varphi^*(P+E))$ for any effective divisor~$E$ in~$\mathcal{F}_\Lambda$ that does not have~$P$ among its irreducible components.
Using \eqref{eq: Heegner decomp} we have
\[{\rm deg}(\varphi^*P^\Lambda_{-m, \mu})=\sum_{\substack{r\in\ZZ_{>0} \\ r^2| m}}\mu(r)\sum_{\substack{\overline\beta\in D(\Lambda)\\r\overline\beta=\mu}}{\rm \deg}(\varphi^*H^\Lambda_{-m/r^2,\overline\beta}),
\]
where $\mu(\cdot)$ is the M\"{o}bius function.
Thus by Proposition~\ref{prop:pulbackheegner} we aim to show that 
\begin{equation}
\label{eq: pullback of heegner}
\sum_{\substack{r\in\ZZ_{>0} \\ r^2| m}}
\mu(r)
\sum_{\substack{\overline\beta\in D(\Lambda)\\r\overline\beta=\mu}}\sum_{\substack{\alpha\in\Lambda/L\oplus K\\t\in \ZZ-q((\alpha+\beta)_L)\\ t\geq 0}}
        \theta\Big(\frac{m}{r^2}-t,(\alpha+\beta)_K\Big){\rm{deg}}\big(
        H_{-t,(\alpha+\beta)_L}^L\big)<0,
\end{equation}
where~$\beta$ is a fixed preimage of~$\overline\beta$ under~$\Lambda^\vee/L\oplus K\to\Lambda^\vee/\Lambda$.

Let $\widetilde{f}_L\colon \widetilde{\mathcal{F}}_L\rightarrow \mathcal{F}_L^{ BB}$ be the map of Diagram \eqref{sec3:diag:alpha-beta} and 
consider the map 
\[
\begin{aligned}
{\rm deg}\colon & {\rm Pic}_\QQ(\mathcal{F}_L)\rightarrow \QQ,\qquad
D\mapsto \overline{D}\cdot  \big(\widetilde{f}_L^*\lambda_L\big)^{n-2}. 
\end{aligned}
\]
The Kudla--Millson theta function yields the corresponding modular form
\begin{equation}\label{eq: deg series}
    \sum_{\ell\in D(L)}
    \sum_{\substack{s\in \ZZ-q(\ell) \\ s\ge 0}}
    {\rm{deg}}(H_{-s,\ell}^L)q^s \mathfrak{e}_\ell\in M^{(n+1)/2}_{1,L}.
\end{equation}
By \cite[Theorem I]{Kud03} (see also \cite[Corollary 4.12]{Kud03}) this is a multiple of the Eisenstein series $E_{\frac{n+1}{2},L}$. Concretely
\begin{equation}
\label{eq: eisenstein multiple}
{\rm{deg}}(H^L_{-s,\ell})=-\chi\cdot c_{s,\ell}(E_{\frac{n+1}{2},L}) \;\;\;\hbox{and}\;\;\;{\rm{deg}}(H^L_{0,0})=-\chi\cdot c_{0,0}(E_{\frac{n+1}{2},L})=-\chi,
\end{equation}
where $c_{s,\ell}\in\big((M^{ (n+1)/2}_{1,L})^\circ(\QQ)\big)^\vee$ is the $(s,\ell)$-coefficient extraction functional, and $\chi={\rm{vol}}(\mathcal{F}_L)$ is a positive constant. Thus combining with \eqref{eq: pullback of heegner} and using that $\chi>0$, we have extremality of $\overline{P}$ provided that
\begin{equation}
\label{sec3:eq:ineqE}
-\sum_{\substack{r\in\ZZ_{>0} \\ r^2| m}}\mu(r)\sum_{\substack{\overline\beta\in D(\Lambda)\\r\overline\beta=\mu}}
\sum_{\substack{\alpha\in\Lambda/L\oplus K\\t\in\ZZ-q((\alpha+\beta)_L)\\ t\geq 0}}
        \theta\Big(\frac{m}{r^2}-t,(\alpha+\beta)_K\Big)c_{t,(\alpha+\beta)_L}\Big(E_{\frac{n+1}{2},L}\Big)<0.
\end{equation}

The Fourier coefficients of $E_{\frac{n+1}{2},L}$ are all negative except for the $c_{0,0}$-coefficient, see \cite{BK01}. Moreover observe that since $K$ is definite of rank one, we know 
\[
0\le \theta\Big(\frac{m}{r^2}-t,(\alpha+\beta)_K\Big)\le 2
\]
and $\theta(m/r^2-t,(\alpha+\beta)_K)=0$ when $t>m/r^2$.
Further, if $t=0$, then $(\alpha+\beta)_{L}$ is isotropic in $D(L)$, in particular $c_{0,(\alpha+\beta)_L}(E_{\frac{n+1}{n},L})=1$ when $(\alpha+\beta)_L=0$ and zero otherwise.
We may then rewrite \eqref{sec3:eq:ineqE} as~\eqref{sec3:eq:bound}.
\end{proof}

As a corollary one obtains the following weaker, yet simpler criterion for extremality:

\begin{corollary}
\label{sec3:coro:extremality3}
Under the same hypothesis as in Theorem \ref{thm: extremality2}, the divisor $P$ is extremal in $\overline{\rm{Eff}}(\mathcal{F}_{\Lambda})$ and the closure $\overline{P}$ extremal in $\overline{\rm{Eff}}(\overline{\mathcal{F}}_{\Lambda})$ if
\begin{equation}\label{eq: simpler ineq}
-\sum_{\substack{\alpha\in\Lambda/L\oplus K\\t\in \ZZ-q(\alpha_L)\\ t> 0}}
        \theta\Big(m-t,\Big(\alpha+\frac{\rho}{\mathrm{div}_\Lambda(\rho)}\Big)_K\Big)
        c_{t, \alpha_L}
        \big(E_{\frac{n+1}{2}, L}\big)<\begin{cases}2\hbox{ if }d_\mu=1,2,\\ 1\hbox{ otherwise, }\end{cases}
\end{equation}
where $d_\mu$ stands for the order of $\mu$ in $D(\Lambda)$.
\end{corollary}

\begin{proof}
Note that from \eqref{eq: Heegner decomp} we have ${\rm deg} (\varphi^* P^\Lambda_{-m, \mu})\le {\rm deg} (\varphi^* H^\Lambda_{-m, \mu})$.
Thus, to establish~\eqref{eq: pullback of heegner} it is enough to show 
\begin{equation}\label{sec3:eq:ineqEbis}
\sum_{\substack{\alpha\in\Lambda/L\oplus K\\t\in \ZZ-q((\alpha+\tilde\mu)_L)\\ t\geq 0}}
        \theta(m-t,(\alpha+\tilde\mu)_K){\rm{deg}}(
        H_{-t,(\alpha+\tilde\mu)_L}^L)<0
\end{equation}
for some preimage~$\tilde\mu$ of~$\mu$ under~$\Lambda^\vee/L\oplus K\to\Lambda^\vee/\Lambda$.
We choose~$\tilde\mu$ to be the class of~$\frac{\rho}{\mathrm{div}_\Lambda(\rho)}$ in~$\Lambda^\vee/L\oplus K$.
Since~$\tilde\mu=\big(0,\frac{\rho}{{\rm{div}}_\Lambda(\rho)}\big)\in D(L)\oplus D(K)$, we have $(\alpha+\tilde\mu)_L=\alpha_L$ for all $\alpha \in \Lambda/L\oplus K$.

Consider the subsum of \eqref{sec3:eq:ineqEbis} given by the terms with~$t=0$, namely
\[
\sum_{\substack{\alpha\in\Lambda/L\oplus K\\q((\alpha+\tilde\mu)_L)= 0}}
        \theta(m-t,(\alpha+\tilde\mu)_K){\rm{deg}}(
        H_{0,(\alpha+\tilde\mu)_L}^L).
\]
Since~$H_{0,(\alpha+\tilde\mu)_L}^L=H_{0,\alpha_L}^L$ is not zero only when~$\alpha_L=0$, the subsum above boils down to
\[
 \underbrace{{\rm{deg}}(
        H_{0,0}^L)}_{=-\vol(\mathcal{F}_L)}
        \sum_{\substack{\alpha\in\Lambda/L\oplus K\\\alpha_L=0}}
        \theta(m,(\alpha+\tilde\mu)_K).
\]
Note also that if~$\alpha\equiv (0, \alpha_K)$ mod~$L\oplus K$, then as an element of $\Lambda_\QQ$ we have $\alpha=v_L+\gamma \rho$, where $v_L\in L$ and $\gamma \in \QQ$. Hence $\alpha-v_L=\gamma \rho$. But then since $\alpha$ and $v_L$ are both in~$\Lambda$, it follows that $\gamma \rho$ must be in $\Lambda$. Since $\rho$ is primitive, $\gamma $ must be an integer and so $\alpha_K=0$. Thus if $\alpha_L=0$, we must also have $\alpha_K=0$ and hence $\alpha=0$ mod $L\oplus K$.
We may then rewrite~\eqref{sec3:eq:ineqEbis} as
\begin{equation}
\label{eq: simpler ineq2}
-\sum_{\substack{\alpha\in\Lambda/L\oplus K\\t\in \ZZ-q((\alpha+\tilde\mu)_L)\\ 0< t\leq m}}
        \theta(m-t,(\alpha+\tilde\mu)_K)c_{t, (\alpha+\tilde\mu)_L}
        \big(E_{\frac{n+1}{2}, L}\big)<
        \theta(m,\tilde\mu_K).
\end{equation}

Finally, since $\rho$ is a primitive representative for $(-m,\mu)$, one has that $\theta(m,\tilde\mu_K)=1$ if~$\tilde\mu_K\neq-\tilde\mu_K$ and $\theta(m,\tilde\mu_K)=2$ otherwise.
And the latter occurs if and only if the order of $\mu$ in $D(\Lambda)$ is one or two, since~$d_\mu=\mathrm{div}_\Lambda(\rho)$.
\end{proof}

\begin{remark}\label{sec3:rmk:d=12}
    We observed that we must always have $0\le \theta(m-t,(\alpha+\rho/\mathrm{div}_\Lambda(\rho))_K)\le 2$. Thus, the simplified Theorem \ref{thm: extremality} stated in the introduction follows immediately from Corollary \ref{sec3:coro:extremality3}. 
\end{remark}

We now use Corollary \ref{sec3:coro:extremality3}, in fact we use the simpler Theorem \ref{thm: extremality}, to prove a criterion ensuring the \emph{simultaneous} extremality of all irreducible components of~$P_{-m,\mu}$.
These bounds will also shed light on the asymptotic behavior of extremality, see Section \ref{sec3:asymptotic}.

\begin{theorem}
\label{thm: numerical bounds}
Let $\Lambda$ be an even lattice of signature $(2,n)$ with $n\ge 4$.
Let $D_\Lambda$ be the discriminant of $\Lambda$ and let $k=\frac{n+1}{2}\in \frac{1}{2}\mathbb{Z}$.
For a given $\mu\in D(\Lambda)$ having order $d_\mu$ and $m\in \mathbb{Z}-q_\Lambda(\mu)$, all irreducible components of~$P_{-m, \mu}^\Lambda$ are extremal in $\overline{{\rm{Eff}}}(\mathcal{F}_\Lambda)$ and their closures are extremal in $\overline{{\rm{Eff}}}(\overline{\mathcal{F}}_\Lambda)$ provided 
\[
\frac{1}{4md_\mu}>\begin{cases}
    \sum_{i=1}^{\lfloor 4m^2d_\mu\rfloor}\frac{(2\pi)^k\zeta(k-1)\zeta(k)}{\sqrt{2mD_\Lambda}\cdot \Gamma(k)\zeta(2k)}\big(\frac{i}{4md_\mu}\big)^{k-1}
    \prod_{p|2mD_\Lambda}(2+2{\rm ord}_p(i))&\mbox{ if } n \text{ odd,}\\
    \sum_{i=1}^{\lfloor 4m^2d_\mu\rfloor}\frac{(2\pi)^k\zeta(k-\frac{1}{2})\sigma_{2-2k}(2mD_\Lambda)\sigma_{1/2-k}(2mD_\Lambda)}{\sqrt{2mD_\Lambda}\cdot \Gamma(k)\zeta(2k-1)}\big(\frac{i}{4md_\mu}\big)^{k-1}
    \prod_{p|2mD_\Lambda}\frac{2+2{\rm ord}_p(i)}{1-p^{1-2k}} &\mbox{ if } n \text{ even.}   
\end{cases}
\]
This is always satisfied if $m^2d_\mu<\frac{1}{4}$.  
\end{theorem}

\begin{proof}
Let $\rho$ be a primitive representative for the pair $(-m,\mu)$. Then ${\rm div}_\Lambda(\rho)=d_\mu$ and~$-2m=\langle \rho, \rho\rangle/d_\mu^2$.
As before, we fix $L=\rho^\perp$ and $K=\mathbb{Z}\rho$. Any $\alpha\in \Lambda$ thus has orthogonal decomposition in $L^\vee\oplus K^\vee$ given by
\begin{equation}
\label{sec3:eq:alpha}
\alpha=\underbrace{\left(\alpha-\frac{\langle\alpha,\rho\rangle}{\langle\rho,\rho\rangle}\rho\right)}_{\alpha_L}+\underbrace{\left(\frac{\langle\alpha,\rho\rangle}{\langle\rho,\rho\rangle}\rho\right)  }_{\alpha_K}.
\end{equation}
Let~$\kappa=\lvert\langle\rho,\rho\rangle\rvert/d_\mu=2md_\mu$.We claim that $\Lambda\big/(L\oplus K)$ is cyclic generated by any $\alpha\in\Lambda$ mod~$ L\oplus K$ such that $d_\mu=\left|\langle \alpha,\rho\rangle\right|$. The inclusion $L\oplus K\subset \Lambda$ has index $\kappa$, see \cite[Chapter 14, Equation~0.2]{Huy16} and \cite[Lemma 7.2]{GHS13},
so it is enough to show that $\alpha$ has order $\kappa$ in $\Lambda\big/(L\oplus K)$. Indeed, if $r\alpha\in L\oplus K$ with $r\mid \kappa$, then from \eqref{sec3:eq:alpha} one notes that $r\cdot\frac{d_\mu}{\langle\rho,\rho\rangle}\rho\in \mathbb{Z}\rho$. In particular $\kappa\mid r$ and therefore $r=\kappa$.

For any $\alpha\in \Lambda /(L\oplus K)$, we know that $\kappa\alpha_L\in L$ and so since $\alpha_L\in L^\vee$, we 
have $2\kappa q_L(\alpha_L)=\langle \kappa \alpha_L, \alpha_L\rangle \in \mathbb{Z}.$
It follows that $q_L(\alpha_L)\in \frac{1}{2\kappa}\mathbb{Z}$. Thus if $t\in \mathbb{Z}-q(\alpha_L)$ and satisfies $0< t\leq m$, we know that
\begin{equation}
\label{eq: t values}
t\in \left\{\frac{i}{2\kappa}\mid  i\in \mathbb{Z},\  0<i\le 4m^2 d_\mu\right\}.
\end{equation}

We remark that this already establishes that under the additional assumption that $m^2d_\mu<\frac{1}{4}$, then the set \eqref{eq: t values} is empty and so the inequality of Theorem \ref{thm: extremality} is automatically satisfied. Thus in this case $P_{m, \mu}^\Lambda$ is extremal.

Given the restriction \eqref{eq: t values} on possible values of $t$, and since all Fourier coefficients $c_{t,\alpha_L}$ with $t>0$ are negative, we have 
\begin{equation}\label{eq: bounding summands}
-\sum_{\substack{\alpha\in \Lambda/L\oplus K\\t\in \ZZ-q(\alpha_L)\\ 0<t\leq m }}
c_{t,\alpha_L}\big(E_{\frac{n+1}{2},L}\big)\le 
-\sum_{\alpha\in\Lambda/L\oplus K}\sum_{i=1}^{\lfloor 4m^2d_\mu\rfloor}
c_{\frac{i}{2\kappa},\alpha_L}\big(E_{\frac{n+1}{2},L}\big).    
\end{equation}

For any discriminant $D\in \mathbb{Z}\backslash \{0\}$ define the Dirichlet character $\chi_{4D}=\left(\frac{D}{a}\right)$, and the divisor sum with character $\sigma_s(a,\chi)=\sum_{d|a}\chi(d)d^s$. Then one has
\begin{equation}\label{eq: dirichlet bound}
\sigma_{-s}(a,\chi)\le \zeta(s)\;\;\;\hbox{and}\;\;\; \ \frac{\zeta(2s)}{\zeta(s)}\le L(s,\chi)\le \zeta(s).
\end{equation}

Now assume that $\boldsymbol{n}$ \textbf{is odd} (so that~$n-1$ is even).
Then by~\cite[Theorem~11]{BK01} together with the bounds of \cite[Lemma 2.3]{BM19} and~\eqref{eq: dirichlet bound}, we have
\[-c_{t,\alpha_L}\big(E_{\frac{n+1}{2}, L}\big)\le C_L|t|^{k-1}\prod_{p|2N}\big(2+2{\rm ord}_p(2d_{\alpha_L}t)\big),\]
where 
\[C_L=\frac{(2\pi)^k\zeta(k-1)\zeta(k)}{\sqrt{D_L}\cdot \Gamma(k)\zeta(2k)}.\]
Here  $D_L$ is the discriminant and $N$ is the level of the lattice $L$. In particular, in light of \eqref{eq: t values}, we have
\begin{equation}
\label{eq: bounding coeffs1}
-c_{\frac{i}{2\kappa},\alpha_L}\big(E_{\frac{n+1}{2}, L}\big)\le C_L\Big(\frac{i}{2\kappa}\Big)^{k-1}
\prod_{p|2N}\bigg( 2+2{\rm ord}_p\bigg(\frac{d_{\alpha_L}}{\kappa}\cdot i\bigg)\bigg).
\end{equation}
Note that since $\kappa \alpha_L=0$, we know that $d_{\alpha_L}$ divides $\kappa$. Thus ${\rm ord}_p\left(\frac{d_{\alpha_L}}{\kappa}\cdot i\right)\le {\rm ord}_p(i)$. Moreover, we know $2N$ divides the discriminant 
$D_L$ of the lattice $L$ given by
\[D_L=\frac{D_K\cdot D_\Lambda}{d_\mu^2}=2m D_\Lambda.\]
Thus we may further bound \eqref{eq: bounding coeffs1} as
\[
-c_{\frac{i}{2\kappa},\alpha_L}\big(E_{\frac{n+1}{2}, L}\big)\le C_L\Big(\frac{i}{2\kappa}\Big)^{k-1}\prod_{p|2mD_\Lambda}\left(2+2{\rm ord}_p\left(i\right)\right).
\]
Recall that $\left|\Lambda\big/L\oplus K\right|=2md_\mu$. Theorem \ref{thm: extremality} together with \eqref{eq: bounding summands} then yield that if 
\begin{equation}
\label{sec3:eq:nodd}
\sum_{i=1}^{\lfloor 4m^2d_\mu\rfloor}C_L\Big(\frac{i}{2\kappa}\Big)^{k-1}\prod_{p|2mD_\Lambda}\left(2+2{\rm ord}_p\left(i\right)\right)<\frac{1}{4md_\mu},
\end{equation}
then all irreducible components of $P_{-m, \mu}^\Lambda$ are extremal. 

Similarly, if $\boldsymbol{n}$ \textbf{is even} (so that $n-1$ is odd) then by 
\cite[Theorem 4.6]{BK01}, \cite[Equation (10), Lemma 2.3]{BM19}, \eqref{eq: dirichlet bound}, and \eqref{eq: t values} we have
\[-c_{\frac{i}{2\kappa},\alpha_L}\big(E_{\frac{n+1}{2}, L}\big)\le C_L'\Big(\frac{i}{2\kappa}\Big)^{k-1}\prod_{p|2mD_\Lambda}\frac{2+2{\rm ord}_p(i)}{1-p^{1-2k}},\]
where 
\[C_L'=\frac{(2\pi)^k\zeta(k-\frac{1}{2})\sigma_{2-2k}(2mD_\Lambda)\sigma_{1/2-k}(2mD_\Lambda)}{\sqrt{D_L}\cdot \Gamma(k)\zeta(2k-1)}.\]
Thus, as with the $n$ odd case, combining with Theorem \ref{thm: extremality} and \eqref{eq: bounding summands} yields that if 
\begin{equation}
\label{sec3:eq:neven}
\sum_{i=1}^{\lfloor 4m^2d_\mu\rfloor}C_L'\Big(\frac{i}{2\kappa}\Big)^{k-1}\prod_{p|2mD_\Lambda}\frac{2+2{\rm ord}_p(i)}{1-p^{1-2k}}<\frac{1}{4md_\mu},
\end{equation}
then all irreducible components of $P_{-m, \mu}^\Lambda$ are extremal.
\end{proof}

\begin{remark}
Note that in both inequalities \eqref{sec3:eq:nodd}, \eqref{sec3:eq:neven}, if $d_\mu=1,2$, then $\frac{1}{4md_\mu}$ can be replaced with $\frac{1}{2md_\mu}$, see Remark \ref{sec3:rmk:d=12}.
\end{remark}

Theorem \ref{intro:thm:main_K3} and Corollary \ref{thm: numerical bounds intro} from the introduction are now an immediate consequence of Theorem \ref{thm: numerical bounds}:

\begin{proof}[Proof of Corollary \ref{thm: numerical bounds intro}]
Assume $m^2\cdot d_\mu<1/4$. From Theorem \ref{thm: numerical bounds} it follows that all components $P$ of $P_{-m,\mu}$ (and their closure) are extremal. Note that from \eqref{eq: Heegner decomp} one has that if $P$ is a component of $H_{-m,\mu}$, then $P$ is a component of $P_{-s,\delta}$ for some $s=m/r^2$ and $r\delta=\mu$. Then, the order $d_\delta$ divides $r\cdot d_\mu$. In particular
\[
s^2d_\delta\le \frac{m^2}{r^4}\cdot r d_\mu=\frac{1}{r^3}\cdot m^2d_\mu<\frac{1}{4}.
\]
Therefore, one obtains extremality for all components $P$ of $H_{-m,\mu}$ and their respective closures.
\end{proof}

\begin{proof}[Proof of Theorem \ref{intro:thm:main_K3}]
For a Noether-Lefschetz divisor $D_{h,a}$ on the moduli space of quasi-polarized K3 surfaces $\mathcal{F}_{2d}$, recall (see Example \ref{ex: NL Heegner relationship} as well as  \cite[Section~1, Lemma~3]{MP13}) that $D_{h,a}= H_{-m, \mu}$ if  $d\nmid a$ and $D_{h,a}=\frac{1}{2}H_{-m, \mu}$ if $d \mid a$, where $m=\frac{a^2}{4d}-(h-1)$ and $\mu=a\cdot \frac{\ell}{2d}$ with $\frac{\ell}{2d}$ the standard generator of $D(\Lambda_{2d})$.  

Thus, if $d,a$ satisfy the bound of Theorem \ref{intro:thm:main_K3}, then the order $d_\mu$ of $\mu=a\cdot \frac{\ell}{2d}$ is $\frac{2d}{{\rm gcd}(a, 2d)}$ 
and so we have $m^2 d_\mu< \frac{1}{4}$. The theorem follows from Corollary \ref{thm: numerical bounds intro}.
\end{proof}

\subsection{Asymptotic behavior of extremality}\label{sec3:asymptotic} In this section we discuss some applications of Theorem \ref{thm: numerical bounds} to the asymptotic behavior of extremality. The first thing to notice is that when $m,d_\mu,$ and $D_{\Lambda}$ are fixed, 
the function dominating the right-hand side of the inequality in Theorem \ref{thm: numerical bounds} as $k=n/2+1$ grows to infinity is $\Gamma(k)$. We call $f(n,D_\Lambda,m,d_\mu)/4md_\mu$ the right-hand side of the inequality in Theorem \ref{thm: numerical bounds}. Assume $n$ is odd. Then there exists a constant $C>0$ depending on $D_\Lambda$, $m$, and $d_\mu$ such that 
\[
f(n,D_\Lambda,m,d_\mu)< C\cdot \frac{(2\pi)^{k}\zeta(k-1)\zeta(k)\cdot m^{k-1}}{\Gamma(k)\zeta(2k)}.
\]
Note that the fraction $\frac{\zeta(k-1)\zeta(k)}{\zeta(2k)}$ converges to $1$ as $k$ grows to infinity. In particular, if we call $M=\max\{2\pi, m\}$, then there exists a constant $C'>0$ such that for $k\gg0$ one has
\begin{equation}
\label{sec3:eq:limk}
f(n,D_\Lambda,m,d_\mu)<C'\frac{(M^2)^k}{\Gamma(k)}\longrightarrow 0, \qquad n\longrightarrow\infty.
\end{equation}
The case when $n$ is even is analogous. As an application, one obtains the following corollary.

\begin{corollary}\label{cor: asymp rank}
Let $\Lambda$ be an even lattice of signature $(2,n)$ with $n\ge 4$.
\begin{enumerate}
    \item If $\Lambda$ is unimodular and~$m\in\ZZ_{>0}$, then the divisor~$P_{-m}^\Lambda$ is extremal in~$\overline{{\rm{Eff}}}(\mathcal{F}_\Lambda)$ and its closure $\overline{P}{}_{-m, \mu}^\Lambda$ is extremal in $\overline{{\rm{Eff}}}(\overline{\mathcal{F}}_\Lambda)$ for any sufficiently large $n$.
    \item For any $\mu\in D(\Lambda)$ and $m\in \ZZ-q(\mu)>0$, all irreducible components of~$P_{-m,\mu}$ on $\mathcal{F}_\Lambda$ are extremal in~$\overline{{\rm{Eff}}}(\mathcal{F}_{\Lambda\oplus M})$ and their closures are extremal in $\overline{{\rm{Eff}}}(\overline{\mathcal{F}}_{\Lambda\oplus M})$ for any even negative-definite lattice $M$ of sufficiently large rank.
\end{enumerate}
\end{corollary}

\begin{proof}
This follows from \eqref{sec3:eq:limk} together with Theorem \ref{thm: numerical bounds}.
\end{proof}

Now we consider the case that $n,m$ and $d_\mu$ are fixed, and the discriminant $D_\Lambda$ grows. First, we treat the case when $\boldsymbol{n}$ \textbf{is odd}. Then there exists a constant $C>0$ depending on $n,m,d_\mu$ such that 
\begin{equation}
\label{sec3:eq:asym_D}
f(n,D_\Lambda,m,d_\mu)<C\cdot \frac{1}{D_{\Lambda}^{1/2}}\cdot \sum_{i=1}^{\lfloor4m^2d_\mu\rfloor}\prod_{p\mid 2mD_{\Lambda}}(2+2{\rm{ord}}_p(i)).
\end{equation}
Note that if $p_0^N$ is the largest power of a prime number in $\{1,\ldots, \lfloor4m^2d_\mu\rfloor\}$, then $N$ depends only on $m$ and $d_\mu$, and for any prime $p$ one has ${\rm{ord}}_p(i)\leq N$. In particular if~$\omega(r)$ denotes the number of prime factors of $r$, then \eqref{sec3:eq:asym_D} may be bounded as
\begin{equation}
\label{sec3:eq:asym_D2}
f(n,D_\Lambda,m,d_\mu)<C\cdot \frac{\lfloor4m^2d_\mu\rfloor}{D_{\Lambda}^{1/2}}\cdot (2+2N)^{\omega(2mD_{\Lambda})}.
\end{equation}
Recall that if $d(r)$ is the divisor function counting the number of divisors of $r$, then for any $\varepsilon>0$ one has~$2^{\omega(r)}\leq d(r)=o(r^\varepsilon)$.
In particular, since $m$ is fixed 
\[
(2+2N)^{\omega(2mD_{\Lambda})}\leq (d(2mD_{\Lambda}))^{\log_2(2+2N)}=o((2mD_{\Lambda})^\varepsilon)=o(D_{\Lambda}^\varepsilon)
\]
and so \eqref{sec3:eq:asym_D2} goes to zero as $D_{\Lambda}$ grows.

Now we consider the case when~$\boldsymbol{n}$ \textbf{is even}.
Since~$1/2<1-p^{1-2k}$, then
\[
\frac{2+2{\rm{ord}}_p(i)}{1-p^{1-2k}}\leq 4+4{\rm{ord}}_p(i).
\]
A similar argument as for~$n$ odd shows that for $n$, $m$ and $d_\mu$ fixed $f(n,D_\Lambda,m,d_\mu)$ is also in $o\big(D_{\Lambda}^{\varepsilon-1/2}\big)$ for any $\varepsilon>0$. In particular, it goes to zero as $D_\Lambda$ grows.

As an example of this phenomenon, in the case of moduli spaces of K3 surfaces, one obtains the following result.

\begin{corollary}\label{cor: K3 asymp}
Let $m\in \mathbb{Q}_{> 0}$ and $d_\mu\in \mathbb{Z}_{> 0}$ fixed. If $d$ is large enough, then for all~$0\leq a<2d$ with $2d/\gcd(a,2d)=d_{\mu}$ the divisor $P_{-m,a\ell_*}$ is either empty or extremal in~$\overline{{\rm{Eff}}}(\mathcal{F}_{2d})$ and its closure extremal in~$\overline{{\rm{Eff}}}(\overline{\mathcal{F}}_{2d})$.
\end{corollary}

\begin{proof}
In this case one has discriminant $D_{\Lambda}=2d$. Since $(2+2N)^{\omega(4dm)}$ is $o(d^\varepsilon)$ for all~$\varepsilon>0$, the right-hand side of \eqref{sec3:eq:asym_D2} is in $o(d^{\varepsilon-1/2})$, in particular it goes to zero as $d$ grows. The corollary then follows from Theorem \ref{thm: numerical bounds}.  
\end{proof}

\begin{remark}\label{rem: k3 asymp}
Observe that for fixed  $m\in \mathbb{Q}_{> 0}$ and $d_\mu\in \mathbb{Z}_{> 0}$ the primitive Heegner divisor $P_{-m, a\ell_*}$ on $\mathcal{F}_{2d}$ is nonempty if and only if  there is an $n\in \ZZ$ such that $-\frac{a^2}{4d}+n=m$. This occurs if and only if $-\frac{a^2}{{\rm gcd}(a,2d)}+2nd_\mu=2md_{\mu}$. Letting $\kappa=2md_{\mu}$ as in the proof of Theorems \ref{thm: extremality2} and \ref{thm: numerical bounds} above, we then have that  $P_{-m, a\ell_*}$ is nonempty if and only if~$-\kappa \cdot {\rm gcd}(a,2d)$ is a square modulo $2d_\mu$.
In particular, for fixed $m\in \mathbb{Q}_{> 0}$ and $d_\mu\in \mathbb{Z}_{> 0}$ there are infinitely many choices of values ${\rm gcd}(a,2d)$ and hence infinitely many choices of $2d=d_{\mu}{\rm gcd}(a,2d)$ for which there exists a class $a \in \ZZ/2d\ZZ$ such that $P_{-m, a\ell_*}$ is nonempty on $\mathcal{F}_{2d}$. Corollary \ref{cor: K3 asymp} implies that there is a fixed bound $d_0$ such that for all (infinitely many) $d\ge d_0$ such that  $P_{-m, a_d\ell_*}$ is nonempty on $\mathcal{F}_{2d}$ (where here we use the notation $a_d$ to denote the fact that this value depends on the choice of $d$), the divisor $P_{-m, a_d\ell_*}$
    is  extremal in~$\overline{{\rm{Eff}}}(\mathcal{F}_{2d})$ and its closure extremal in~$\overline{{\rm{Eff}}}(\overline{\mathcal{F}}_{2d})$.
\end{remark}

\subsection{Extremality via pullbacks of Siegel Eisenstein series}\label{sec;extr pullback Siegel}
In this section, we work under the additional assumption that~$\Lambda$ splits off \textbf{two hyperbolic planes}, so that~$P_{-m,\mu}$ is irreducible.
We illustrate here another approach to checking the extremality of the ray generated by~$P_{-m,\mu}$, this time by computing~$\vol(P_{-m,\mu}^2)$.

We denote by $\omega$ the K\"{a}hler form given by $c_1(\lambda)\in H^{1,1}(\mathcal{F}_{\Lambda}, \mathbb{R})$.
By~\cite[Theorem~3.2]{Mum77} both $\omega$ and the Hodge bundle $\lambda$ extend to any toroidal compactification $\mathcal{F}_{\Lambda}^{\rm tor}$, making the latter of finite volume.
One defines for any cycle $Y\in H_{2k}\left(\mathcal{F}_{\Lambda},\mathbb{Q}\right)$ 
\begin{equation}
\label{sec3:eq:vol_def}
\vol(Y)=\int_{Y}\omega^{n-k}=\int_{\overline{Y}}\overline{\omega}^{n-k}.
\end{equation}

We illustrate here a formula for~$\vol(P_{-m,\mu}^2)$.
The proof is based on the pullback formula of the Siegel Eisenstein series~$E^k_{2,\Lambda}$ with respect to the diagonal embedding~$\HH\times\HH\to\HH_2$, $(\tau_1,\tau_2)\mapsto\big(\begin{smallmatrix}
    \tau_1 & 0\\
    0 & \tau_2
\end{smallmatrix}\big)$.
Formulas for diagonal restrictions of Siegel Eisenstein series are nowadays classical.
They were proved by Garret~\cite{garrett} and Böcherer~\cite{boecherer}, and led to the so-called doubling method.

The reason  to consider the volume of the self-intersection of primitive Heegner divisors is the following.
\begin{lemma}\label{lemma:fromdegtovol}
    Let $\Lambda$ be an even lattice of signature $(2,n)$ with $n\ge 4$ splitting off two copies of the hyperbolic plane.
    Let~$\mu\in D(\Lambda)$ and let~$m\in\ZZ-q(\mu)>0$.
    If $\vol(P^2_{-m,\mu})<0$, then~$P_{-m,\mu}$ generates an extremal ray in the pseudoeffective cone~$\overline{{\rm Eff}}(\mathcal{F}_\Lambda)$ and its closure~$\overline{P}_{-m,\mu}$ an extremal ray in~$\overline{{\rm Eff}}(\overline{\mathcal{F}}_\Lambda)$.
\end{lemma}
\begin{proof}
Let~$\rho$ be a primitive representative of~$(-m,\mu)$ and, $L=\rho^\perp\subset\Lambda$, and consider the corresponding map~$\varphi\colon\mathcal{F}_L\to\mathcal{F}_\Lambda$. We keep the same notation as in the proof of Corollary \ref{cor: degree condition}.
As in such proof, we can find a representative $C$ of the curve class $[C]=\big(\widetilde{f}_L^*r\lambda_L\big)^{n-2}\in {\rm{N}}_1\big(\widetilde{\mathcal{F}_{L}}\big)$ which does not meet the boundary of $\widetilde{\mathcal{F}_{L}}$, where $r\lambda_L\subset \mathcal{F}_L$ is the hyperplane class coming from $\mathcal{F}_L^{BB}\subset\mathbb{P}^N$. Note moreover that from Diagram~\eqref{sec3:diag:alpha-beta}, we have~$\widetilde{f}_L^*\lambda_L=\widetilde{\varphi}^*f_\Lambda^*\lambda_\Lambda$.  Thus, using the projection formula as in the proof of Corollary~\ref{cor: degree condition}, for any divisor $D\in {\rm{Pic}}_\mathbb{Q}(\mathcal{F}_{\Lambda})$ we have
\[
{\rm{deg}}(\varphi^*D)=[C]\cdot \widetilde{\varphi}^*\overline{D}
=
\widetilde{\varphi}_*\big(\widetilde{\varphi}^*\left(f^*_\Lambda r\lambda_\Lambda\right)^{n-2}\cdot \widetilde{\varphi}^*\overline{D}\big)
=
r^{n-2}\widetilde{\varphi}_*\big[\widetilde{\mathcal{F}}_L\big]\cdot \left(f^*_\Lambda\lambda_\Lambda\right)^{n-2}\cdot\overline{D}.
\]
When $D=P_{-m,\mu}$, one has $\widetilde{\varphi}_*\big[\widetilde{\mathcal{F}}_L\big]={\rm{deg}}(\varphi)\cdot \overline{P}_{-m,\mu}$. Hence we deduce
\begin{align*}
{\rm{deg}}(\varphi^*P_{-m,\mu})&=\left(r^{n-2}{\rm{deg}}(\varphi)\right)\overline{P}_{-m,\mu}^2\cdot (f_\Lambda^*\lambda_\Lambda)^{n-2}\\
&=\left(r^{n-2}{\rm{deg}}(\varphi)\right)\int_{\overline{P}_{-m,\mu}^2}c_1(f_\Lambda^*\lambda_\Lambda)^{n-2}\\
&=\left(r^{n-2}{\rm{deg}}(\varphi)\right)\int_{P_{-m,\mu}^2}\omega^{n-2}\\
&=\left(r^{n-2}{\rm{deg}}(\varphi)\right)\cdot{\rm{vol}}\big(P_{-m,\mu}^2\big).
\end{align*}
The result then follows from Corollary \ref{cor: degree condition}.
\end{proof}

In light of Lemma \ref{lemma:fromdegtovol}, we now illustrate a formula to compute~$\vol(P^2_{-m,\mu})$.
 Let $k\coloneqq\rk\Lambda/2$, and denote by~$\mathcal{P}^\Lambda_{m,\mu}$ the weight~$k$ Poincaré series of genus~$1$ for~$\weil{\Lambda}$ of indices~$m\in\ZZ-q(\mu)$ and~$\mu\in D(\Lambda)$.
These are constructed similarly as for~$E^k_{1,\Lambda}$, where in the defining series of the latter one replaces~$\mathfrak{e}_0$ with~$e(m\gamma\cdot\tau)\mathfrak{e}_\mu$, see~\cite[Section~1.2]{bruinier-habilitation} for details.
Here~$e(t)\coloneqq e^{2\pi i t}$ as in Section~\ref{sec;vv Siegel mod}.

\begin{proposition}\label{prop:volprHeegner Siegel}
    Let~$\mu\in D(\Lambda)$ and let~$m\in\ZZ-q(\mu)$ be positive.
    Then
    \begin{align*}
        \frac{\vol(P^2_{-m,\mu})}{\vol(\mathcal{F}_\Lambda)}
        &=
        g\big(m,\mu,E^k_{1,\Lambda}\big)^2 + 2\sum_{\substack{r\in\ZZ_{>0} \\ r^2|m}}\mu(r)
        \sum_{\substack{\alpha\in D(\Lambda) \\ r\alpha=\mu}}
        \sum_{\substack{w_1\in\ZZ_{>0} \\ w_1^2|(m/r^2)}}
        \sum_{\substack{\beta\in D(\Lambda) \\ w_1\beta=\alpha}}
        c_{m/r^2w_1^2,\beta}(E^k_{1,\Lambda})
        \\
        &\quad\times
        \sum_{\substack{w_2\in\ZZ_{>0} \\ \gcd(w_1,w_2)=1}}
        g\big(m,\mu,\mathcal{P}^\Lambda_{mw_2^2/r^2w_1^2,w_2\beta}\big)
    \end{align*}
    where~$\mu(\cdot)$ is the Möbius function and~$g$ is the auxiliary function defined as
    \[
    g(m,\mu,f)=\sum_{\substack{r\in\ZZ_{>0} \\ r^2|m}}\mu(r)
    \sum_{\substack{\alpha\in D(\Lambda) \\ r\alpha=\mu}}
    c_{m/r^2,\alpha}(f)
    \qquad
    \text{for all~$f\in M^k_{1,\Lambda}$.}
    \]
\end{proposition}
\begin{proof}
    By~\eqref{eq: Heegner decomp}, we deduce that
    \begin{equation}\label{eq;vol self in proof}
    \vol(P_{-m,\mu}^2)
    =
    \sum_{\substack{r_1,r_2\in\ZZ_{>0} \\ r_i^2|m}}
    \mu(r_1)\mu(r_2)
    \sum_{\substack{\discelone,\disceltwo\in D(\Lambda) \\ r_i \discelgen=\mu}}
    \vol\big(H_{-m/r_1^2,\discelone}\cdot H_{-m/r_2^2,\disceltwo}\big).
    \end{equation}
    We now prove a formula for the volume of the intersection of two Heegner divisors.
    
    The class~$[H_{-{m_1},\discel_1}]\smile[H_{-m_2,\discel_2}]$ is the Fourier coefficient of index~$((m_1,m_2),(\discel_1,\discel_2))$ of~$[\KMtheta{1}\wedge\KMtheta{2}]$, where the latter is considered as a holomorphic function~$\HH\times\HH\to H^4(\mathcal{F}_\Lambda,\CC)\otimes\CC[D(\Lambda)^2]$ that behaves as a modular form of weight~$k=1+n/2$ for~$\weil{\Lambda}$ with respect to each symplectic variable.
    The volume~$\vol(H_{-{m_1},\discel_1}\cdot H_{-m_2,\discel_2})$ is then the Fourier coefficient of index~$((m_1,m_2),(\discel_1,\discel_2))$ of
    \[
    \int_{\mathcal{F}_\Lambda}\KMtheta{1}\wedge\KMtheta{1}\wedge\omega^{n-2}.
    \]

    By~\cite[(9.2)]{kudla-algcycles} we know that~$\KMtheta{1}(\tau_1)\wedge\KMtheta{1}(\tau_2)
    =
    \KMtheta{2}\big(\begin{smallmatrix}
        \tau_1 & 0 \\
        0 & \tau_2
    \end{smallmatrix}\big)$ for all~$\tau_1,\tau_2\in\HH$.
    By~\cite[Corollary~4.24]{Kud03}, see~\cite[Theorem~4.2]{bruinierzuffetti} for a statement in our vector-valued setting, the generating series of volumes of special cycles is an Eisenstein series, and in fact
    \[
    \int_{\mathcal{F}_\Lambda}\KMtheta{1}(\tau_1)\wedge\KMtheta{1}(\tau_2)\wedge\omega^{n-2}
    =
    \vol(\mathcal{F}_\Lambda)
    \cdot E^k_{2,\Lambda}\big(
    \begin{smallmatrix}
        \tau_1 & 0 \\
        0 & \tau_2
    \end{smallmatrix}
    \big).
    \]
    By~\cite[Satz~12]{boecherer} and~\cite[Theorem~2.20]{bruinierzuffetti}, the pullback of~$E^k_{2,\Lambda}$ is
    \begin{align*}
        E^k_{2,\Lambda}\big(
    \begin{smallmatrix}
        \tau_1 & 0 \\
        0 & \tau_2
    \end{smallmatrix}
    \big)
    &=
    E^k_{1,\Lambda}(\tau_1)\otimes E^k_{1,\Lambda}(\tau_2)
    + 2 \sum_{\discelone\in D(\Lambda)}
    \sum_{\substack{m\in\ZZ - q(\discelone) \\ m>0}}
    c_{m,\discelone}(E^k_{1,\Lambda})
    \\
    &\quad \times
    \sum_{w_1\in \ZZ_{>0}}
    e(m w_1^2 \tau_1)
    \mathfrak{e}_{w_1\discelone}
    \otimes\sum_{\substack{w_2\in\ZZ_{>0} \\ \gcd(w_1,w_2)=1}}
    \mathcal{P}^\Lambda_{mw_2^2,w_2\discelone}(\tau_2).
    \end{align*}
    Therefore, the Fourier coefficient of index~$((m_1,m_2),(\discelone,\disceltwo))$ of~$E^k_{2,\Lambda}\big(
    \begin{smallmatrix}
        \tau_1 & 0 \\
        0 & \tau_2
    \end{smallmatrix}
    \big)$ is
    \begin{equation}\label{eq;Fcoefrestrsiegel}
        \begin{split}
            & c_{m_1,\discelone}(E^k_{1,\Lambda}) c_{m_2,\disceltwo}(E^k_{1,\Lambda})
            +
            2\sum_{\substack{w_1\in\ZZ_{>0} \\ w_1^2|m_1}}
            \sum_{\substack{\beta_1\in D(\Lambda) \\ w_1\beta_1=\discelone}}
            c_{m_1/w_1^2,\beta_1}(E^k_{1,\Lambda})
            \\
            &\qquad\times
            \sum_{\substack{w_2\in\ZZ_{>0} \\ \gcd(w_1,w_2)=1}}
            c_{m_2,\disceltwo}\big(
            \mathcal{P}^\Lambda_{m_1w_2^2/w_1^2,w_2\beta_1}
            \big).
        \end{split}
    \end{equation}
    This equals~$\vol(H_{-m_1,\discelone}\cdot H_{-m_2,\disceltwo})/\vol(\mathcal{F}_\Lambda)$.
    It is easy to replace this in~\eqref{eq;vol self in proof} to deduce the claim.
\end{proof}

\begin{remark}
    Combining Proposition \ref{prop:volprHeegner Siegel} with Lemma \ref{lemma:fromdegtovol} then gives a numerical condition for extremality of $P_{-m, \mu}$ analogous to Theorem \ref{thm: extremality2} using this pullback of Siegel Eisenstein series method. 
\end{remark}

\section{Cones of Heegner Divisors}
For $\Lambda$ an even lattice of signature $(2,n)$ with $n\ge 3$ splitting off two copies of the hyperbolic plane and $\mathcal{F}_\Lambda$ its associated orthogonal modular variety, the NL cone ${\rm{Eff}}^{NL}(\mathcal{F}_\Lambda)\subset {\rm{Pic}}_{\mathbb{Q}}(X)$ is the convex cone of effective $\mathbb{Q}$-linear combinations of irreducible components of Heegner divisors (known as primitive Heegner divisors) on $\mathcal{F}_{\Lambda}$.
After tensoring with $\mathbb{R}$, the NL cone ${\rm{Eff}}^{NL}(\mathcal{F}_{\Lambda})$ forms a natural subcone of the cone of pseudo-effective divisors $\overline{{\rm Eff}}(\mathcal{F}_{\Lambda})$. 

The study of NL cones was initiated in \cite{Pet15} in the case of $\mathcal{F}_{2d}$, where the question was raised of whether the containment 
\begin{equation}\label{eq: containment Pet}{\rm Eff}^{NL}(\mathcal{F}_\Lambda)\subset\overline{{\rm Eff}}(\mathcal{F}_\Lambda)\end{equation}
is in fact an equality. 
Bruinier--M\"oller \cite{BM19} showed that for an arbitary orthogonal modular variety~$\mathcal{F}_\Lambda$ as above, the cone ${\rm{Eff}}^{NL}(X)$ is always polyhedral. Thus in particular, if one does have equality in \eqref{eq: containment Pet}, it follows that the pseudo-effective cone $\overline{{\rm Eff}}(\mathcal{F}_\Lambda)$ is polyhedral. 

While the extremality conditions of this paper do not shed light on the above question for the orthogonal modular variety $\mathcal{F}_{2d}$, they do for the orthogonal modular variety~$\mathcal{F}_{\Lambda}$ partially compactifying the moduli space $\mathcal{M}_{{\rm{Kum}}_{2},2}^1$ parameterizing (quasi)-polarized hyperk\"{a}hler fourfolds of Kummer type with split polarization of degree two, which we now explain. 

A primitively polarized hyperk\"ahler manifold $(X,H)$ is of ${\rm Kum}_n$-type if $X$ is deformation equivalent to a fiber of the addition map $A^{[n+1]}\longrightarrow A$ on an abelian surface.
The moduli spaces $\mathcal{M}_{{\rm{Kum}}_{n},2d}^\gamma$ parameterize (quasi)-polarized hyperk\"{a}hler manifolds of ${\rm Kum}_n$-type with polarization of degree $2d$ and divisibility $\gamma$.
In the case  $d=1$ and $\gamma=1$ resp.~$2$, these moduli spaces are partially compactified by the orthogonal modular variety~$\mathcal{F}_\Lambda$ (see \cite[Lemma 4.7]{BBFW24}), where 
$\Lambda=U^{\oplus 2}\oplus Q_d$ (resp.~$Q_t$ with $d=4t-(n+1)$) and
\[
Q_d=\mathbb{Z}\ell+\mathbb{Z}\delta=\left(\begin{array}{cc}-2d&0\\0&-2(n+1)\end{array}\right)\;\;\;\hbox{and}\;\;\;Q_t=\mathbb{Z} u+ \mathbb{Z}v=\left(\begin{array}{cc}-2t&(n+1)\\(n+1)&-2(n+1)\end{array}\right).
\]

When $n=1$, $d=1$, and $\gamma=2$, we have ${\rm Pic}_\QQ\left(\mathcal{F}_\Lambda\right)=\langle \lambda\rangle$ and so the equality ${\rm Eff}^{NL}(\mathcal{F}_\Lambda)=\overline{{\rm Eff}}(\mathcal{F}_\Lambda)$ holds trivially.  By contrast, when $n=1$, $d=1$, and $\gamma=1$, we have $\dim {\rm Pic}_\QQ\left(\mathcal{F}_\Lambda\right)=2$ and the extremality criterion of Theorem \ref{thm: extremality} implies (in fact the criterion $m^2d_\mu<\frac{1}{4}$ of Theorem \ref{thm: numerical bounds} suffices), as stated in Corollary \ref{prop: cone equality} of the introduction, the (non-trivial) equality 
\[
{\rm Eff}^{NL}(\mathcal{F}_\Lambda)=\overline{{\rm Eff}}(\mathcal{F}_\Lambda)
\]
in this case as well.

\begin{proof}[Proof of Corollary \ref{prop: cone equality}]
When $n=2$, $d=1$, and $\gamma=1$, we have $D(\Lambda)=\ZZ/2\times \ZZ/6$ generated by $\ell/2$ and $\delta/6$, where $d_{\ell/2}=2$ and $d_{\delta/6}=6$.
Thus the two primitive Heegner divisors $P_{-\frac{1}{4}, \frac{\ell}{2}}$ and $P_{-\frac{1}{12}, \frac{\delta}{6}}$ both satisfy the criterion for extremality $m^2d_\mu<\frac{1}{4}$ of Theorem~\ref{thm: numerical bounds}.
Since $\overline{{\rm{Eff}}}(\mathcal{F}_\Lambda)\subset {\rm{Pic}}_{\mathbb{R}}(X)$ and $\dim {\rm Pic}_{\mathbb{R}}(\mathcal{F}_\Lambda)=2$, it follows that
\begin{equation}\label{eq: NL cone equality}{\rm Eff}^{NL}(\mathcal{F}_\Lambda)=\overline{{\rm Eff}}(\mathcal{F}_\Lambda)
=
\left\langle P_{-\frac{1}{4}, \frac{\ell}{2}},P_{-\frac{1}{12}, \frac{\delta}{6}}\right\rangle_{\mathbb{R}_{\geq0}}.\qedhere
\end{equation}    
\end{proof}

\begin{remark}
   The description of ${\rm Eff}^{NL}(\mathcal{F}_\Lambda)$ in \eqref{eq: NL cone equality} was originally conjectured in \cite[Remark 4.12]{BBFW24}, but could not be proved using the bounding methods of that paper because the bounds obtained were too large to be computationally feasible. The calculation of $\overline{{\rm{Eff}}}(\mathcal{F}_\Lambda)$ and the resulting equality ${\rm Eff}^{NL}(\mathcal{F}_\Lambda)=\overline{{\rm Eff}}(\mathcal{F}_\Lambda)$ circumvent the need to directly calculate ${\rm Eff}^{NL}(\mathcal{F}_\Lambda)$ in this case. 
\end{remark}


\newpage
\printbibliography

@book {AMRT10,
    AUTHOR = {Ash, A. and Mumford, D. and Rapoport, M. and Tai, Y.-S.},
     TITLE = {Smooth compactifications of locally symmetric varieties},
    SERIES = {Cambridge Mathematical Library},
   EDITION = {Second},
      NOTE = {With the collaboration of Peter Scholze},
 PUBLISHER = {Cambridge University Press, Cambridge},
      YEAR = {2010},
     PAGES = {x+230},
      ISBN = {978-0-521-73955-9},
   MRCLASS = {14M27 (32J05 32M15)},
       DOI = {10.1017/CBO9780511674693},
       URL = {https://doi.org/10.1017/CBO9780511674693},
}

@online {barrosflapanzuffetti-prog,
  author = {Barros, I. and Flapan, L. and Zuffetti, R.},
  title = {\texttt{extremal\_rays}},
  note = {A {S}age{M}ath program to compute modular cones, available on the webpage of the authors together with an explanatory readme file},
  date = {2025},
}

@article {boecherer,
    AUTHOR = {Böcherer, S.},
     TITLE = {{\"U}ber die {F}ourier-{J}acobi-{E}ntwicklung {S}iegelscher
              {E}isensteinreihen},
   JOURNAL = {Math. Z.},
  FJOURNAL = {Mathematische Zeitschrift},
    VOLUME = {183},
      YEAR = {1983},
    NUMBER = {1},
     PAGES = {21--46},
      ISSN = {0025-5874,1432-1823},
   MRCLASS = {11F46 (32N10)},
  MRNUMBER = {701357},
MRREVIEWER = {Martin\ L.\ Karel},
       DOI = {10.1007/BF01187213},
       URL = {https://doi.org/10.1007/BF01187213},
}

@book {bruinier-habilitation,
    AUTHOR = {Bruinier, J. H.},
     TITLE = {Borcherds products on {O}(2, {$l$}) and {C}hern classes of
              {H}eegner divisors},
    SERIES = {Lecture Notes in Mathematics},
    VOLUME = {1780},
 PUBLISHER = {Springer-Verlag, Berlin},
      YEAR = {2002},
     PAGES = {viii+152},
      ISBN = {3-540-43320-1},
   MRCLASS = {11F55 (11F23 11F27 11G18)},
  MRNUMBER = {1903920},
MRREVIEWER = {Rainer\ Schulze-Pillot},
       DOI = {10.1007/b83278},
       URL = {https://doi.org/10.1007/b83278},
}

@misc{bruinierzuffetti,
      title={Lefschetz decompositions of {K}udla-{M}illson theta functions}, 
      author={J. H. Bruinier and R. Zuffetti},
      year={2024},
      eprint={2406.19921},
      archivePrefix={arXiv},
      primaryClass={math.NT},
      note={\url{https://arxiv.org/abs/2406.19921}}, 
}

@article {CC14,
    AUTHOR = {Chen, D. and Coskun, I.},
     TITLE = {Extremal effective divisors on {$\overline{\mathcal{M}}_{1,n}$}},
   JOURNAL = {Math. Ann.},
  FJOURNAL = {Mathematische Annalen},
    VOLUME = {359},
      YEAR = {2014},
    NUMBER = {3-4},
     PAGES = {891--908},
      ISSN = {0025-5831,1432-1807},
   MRCLASS = {14H10 (14C20 14E30 14L30)},
  MRNUMBER = {3231020},
MRREVIEWER = {Flaminio\ Flamini},
       DOI = {10.1007/s00208-014-1027-5},
       URL = {https://doi.org/10.1007/s00208-014-1027-5},
}

@article {funke,
    AUTHOR = {Funke, J.},
     TITLE = {Heegner divisors and nonholomorphic modular forms},
   JOURNAL = {Compositio Math.},
  FJOURNAL = {Compositio Mathematica},
    VOLUME = {133},
      YEAR = {2002},
    NUMBER = {3},
     PAGES = {289--321},
      ISSN = {0010-437X,1570-5846},
   MRCLASS = {11F37 (11F27 11F30)},
  MRNUMBER = {1930980},
MRREVIEWER = {Jan\ Hendrik\ Bruinier},
       DOI = {10.1023/A:1020002121978},
       URL = {https://doi.org/10.1023/A:1020002121978},
}

@incollection {garrett,
    AUTHOR = {Garrett, P. B.},
     TITLE = {Pullbacks of {E}isenstein series; applications},
 BOOKTITLE = {Automorphic forms of several variables ({K}atata, 1983)},
    SERIES = {Progr. Math.},
    VOLUME = {46},
     PAGES = {114--137},
 PUBLISHER = {Birkh\"auser Boston, Boston, MA},
      YEAR = {1984},
      ISBN = {0-8176-3172-0},
   MRCLASS = {11F46},
  MRNUMBER = {763012},
MRREVIEWER = {H.\ Klingen},
}

@article {kudla-algcycles,
    AUTHOR = {Kudla, S. S.},
     TITLE = {Algebraic cycles on {S}himura varieties of orthogonal type},
   JOURNAL = {Duke Math. J.},
  FJOURNAL = {Duke Mathematical Journal},
    VOLUME = {86},
      YEAR = {1997},
    NUMBER = {1},
     PAGES = {39--78},
      ISSN = {0012-7094,1547-7398},
   MRCLASS = {11F32 (11F30 11G18 14C25 14G35)},
  MRNUMBER = {1427845},
MRREVIEWER = {Dipendra\ Prasad},
       DOI = {10.1215/S0012-7094-97-08602-6},
       URL = {https://doi.org/10.1215/S0012-7094-97-08602-6},
}

@article {kudlamillson-harmonicI,
    AUTHOR = {Kudla, S. S. and Millson, J. J.},
     TITLE = {The theta correspondence and harmonic forms. {I}},
   JOURNAL = {Math. Ann.},
  FJOURNAL = {Mathematische Annalen},
    VOLUME = {274},
      YEAR = {1986},
    NUMBER = {3},
     PAGES = {353--378},
      ISSN = {0025-5831,1432-1807},
   MRCLASS = {11F27 (22E41 32N15)},
  MRNUMBER = {842618},
MRREVIEWER = {Yue\ Lin L. Tong},
       DOI = {10.1007/BF01457221},
       URL = {https://doi.org/10.1007/BF01457221},
}

@article {kudlamillson-harmonicII,
    AUTHOR = {Kudla, S. S. and Millson, J. J.},
     TITLE = {The theta correspondence and harmonic forms. {II}},
   JOURNAL = {Math. Ann.},
  FJOURNAL = {Mathematische Annalen},
    VOLUME = {277},
      YEAR = {1987},
    NUMBER = {2},
     PAGES = {267--314},
      ISSN = {0025-5831,1432-1807},
   MRCLASS = {11F27 (11F75 22E40 22E45 32N10)},
  MRNUMBER = {886423},
MRREVIEWER = {Yue\ Lin L. Tong},
       DOI = {10.1007/BF01457364},
       URL = {https://doi.org/10.1007/BF01457364},
}

@article {kudlamillson-intersection,
    AUTHOR = {Kudla, S. S. and Millson, J. J.},
     TITLE = {Intersection numbers of cycles on locally symmetric spaces and
              {F}ourier coefficients of holomorphic modular forms in several
              complex variables},
   JOURNAL = {Inst. Hautes \'Etudes Sci. Publ. Math.},
  FJOURNAL = {Institut des Hautes \'Etudes Scientifiques. Publications
              Math\'ematiques},
    NUMBER = {71},
      YEAR = {1990},
     PAGES = {121--172},
      ISSN = {0073-8301,1618-1913},
   MRCLASS = {11F32 (11F30 11F46 11F67 32N10 32N15)},
  MRNUMBER = {1079646},
       URL = {http://www.numdam.org/item?id=PMIHES_1990__71__121_0},
}

@article {ADL23,
    AUTHOR = {Ascher, K. and DeVleming, K. and Liu, Y.},
     TITLE = {K-stability and birational models of moduli of quartic {K}3
              surfaces},
   JOURNAL = {Invent. Math.},
  FJOURNAL = {Inventiones Mathematicae},
    VOLUME = {232},
      YEAR = {2023},
    NUMBER = {2},
     PAGES = {471--552},
      ISSN = {0020-9910,1432-1297},
   MRCLASS = {14J28 (14J10 14J45)},
  MRNUMBER = {4574660},
MRREVIEWER = {Guolei\ Zhong},
       DOI = {10.1007/s00222-022-01170-5},
       URL = {https://doi.org/10.1007/s00222-022-01170-5},
}

@article {BB66,
    AUTHOR = {Baily, Jr., W. L. and Borel, A.},
     TITLE = {Compactification of arithmetic quotients of bounded symmetric
              domains},
   JOURNAL = {Ann. of Math. (2)},
  FJOURNAL = {Annals of Mathematics. Second Series},
    VOLUME = {84},
      YEAR = {1966},
     PAGES = {442--528},
      ISSN = {0003-486X},
   MRCLASS = {32.65},
MRREVIEWER = {A. Kor\'{a}nyi},
       DOI = {10.2307/1970457},
       URL = {https://doi.org/10.2307/1970457},
}

@misc{BBFW24,
      title={Cones of Noether-Lefschetz divisors and moduli spaces of hyperk\"ahler manifolds}, 
      author={Barros, I. and Beri, P. and Flapan, L. and Williams, B.},
      year={2025},
      note={In: {\it{Math. Ann.}} (to appear). \url{https://arxiv.org/abs/2407.07622}}, 
}

@article {BM24,
    AUTHOR = {Barros, I. and Mullane, S.},
     TITLE = {The {K}odaira classification of the moduli of hyperelliptic
              curves},
   JOURNAL = {J. Algebraic Geom.},
  FJOURNAL = {Journal of Algebraic Geometry},
    VOLUME = {34},
      YEAR = {2025},
    NUMBER = {4},
     PAGES = {685--753},
      ISSN = {1056-3911,1534-7486},
   MRCLASS = {14H10 (14E05)},
  MRNUMBER = {4940396},
}

@article {BLMM17,
    AUTHOR = {Bergeron, N. and Li, Z. and Millson, J. and
              Moeglin, C.},
     TITLE = {The {N}oether-{L}efschetz conjecture and generalizations},
   JOURNAL = {Invent. Math.},
  FJOURNAL = {Inventiones Mathematicae},
    VOLUME = {208},
      YEAR = {2017},
    NUMBER = {2},
     PAGES = {501--552},
      ISSN = {0020-9910},
   MRCLASS = {14J28 (14C22 14J15)},
MRREVIEWER = {Alan Matthew Thompson},
       DOI = {10.1007/s00222-016-0695-z},
       URL = {https://doi.org/10.1007/s00222-016-0695-z},
}

@article {borcherds-GKZ,
    AUTHOR = {Borcherds, R. E.},
     TITLE = {The {G}ross-{K}ohnen-{Z}agier theorem in higher dimensions},
   JOURNAL = {Duke Math. J.},
  FJOURNAL = {Duke Mathematical Journal},
    VOLUME = {97},
      YEAR = {1999},
    NUMBER = {2},
     PAGES = {219--233},
      ISSN = {0012-7094,1547-7398},
   MRCLASS = {11F55 (11F30 11F50 11G18)},
  MRNUMBER = {1682249},
MRREVIEWER = {Rainer\ Schulze-Pillot},
       DOI = {10.1215/S0012-7094-99-09710-7},
       URL = {https://doi.org/10.1215/S0012-7094-99-09710-7},
}

@article {Bor98,
    AUTHOR = {Borcherds, R. E.},
     TITLE = {Automorphic forms with singularities on {G}rassmannians},
   JOURNAL = {Invent. Math.},
  FJOURNAL = {Inventiones Mathematicae},
    VOLUME = {132},
      YEAR = {1998},
    NUMBER = {3},
     PAGES = {491--562},
      ISSN = {0020-9910,1432-1297},
   MRCLASS = {11F37 (11F22 14J28 17B67 57R57)},
MRREVIEWER = {I.\ Dolgachev},
       DOI = {10.1007/s002220050232},
       URL = {https://doi.org/10.1007/s002220050232},
}

@article {BK01,
    AUTHOR = {Bruinier, J. H. and Kuss, M.},
     TITLE = {Eisenstein series attached to lattices and modular forms on
              orthogonal groups},
   JOURNAL = {Manuscripta Math.},
  FJOURNAL = {Manuscripta Mathematica},
    VOLUME = {106},
      YEAR = {2001},
    NUMBER = {4},
     PAGES = {443--459},
      ISSN = {0025-2611,1432-1785},
   MRCLASS = {11F55 (11F30)},
MRREVIEWER = {V.\ Kumar\ Murty},
       DOI = {10.1007/s229-001-8027-1},
       URL = {https://doi.org/10.1007/s229-001-8027-1},
}

@article {Bru02b,
    AUTHOR = {Bruinier, J. H.},
     TITLE = {On the rank of {P}icard groups of modular varieties attached
              to orthogonal groups},
   JOURNAL = {Compositio Math.},
  FJOURNAL = {Compositio Mathematica},
    VOLUME = {133},
      YEAR = {2002},
    NUMBER = {1},
     PAGES = {49--63},
      ISSN = {0010-437X,1570-5846},
   MRCLASS = {11F55 (11G18 14C22 14G35)},
MRREVIEWER = {Rainer\ Schulze-Pillot},
       DOI = {10.1023/A:1016357029843},
       URL = {https://doi.org/10.1023/A:1016357029843},
}

@article {BM19,
    AUTHOR = {Bruinier, J. H. and M\"{o}ller, M.},
     TITLE = {Cones of {H}eegner divisors},
   JOURNAL = {J. Algebraic Geom.},
  FJOURNAL = {Journal of Algebraic Geometry},
    VOLUME = {28},
      YEAR = {2019},
    NUMBER = {3},
     PAGES = {497--517},
      ISSN = {1056-3911},
   MRCLASS = {14J15 (11F55 14G35 14J10 14J28)},
MRREVIEWER = {Zhiyuan Li},
       DOI = {10.1090/jag/734},
       URL = {https://doi.org/10.1090/jag/734},
}

@article {CR91,
    AUTHOR = {Chang, M.-C. and Ran, Z.},
     TITLE = {On the slope and {K}odaira dimension of {$\overline M_g$} for
              small {$g$}},
   JOURNAL = {J. Differential Geom.},
  FJOURNAL = {Journal of Differential Geometry},
    VOLUME = {34},
      YEAR = {1991},
    NUMBER = {1},
     PAGES = {267--274},
      ISSN = {0022-040X,1945-743X},
   MRCLASS = {14H10},
  MRNUMBER = {1114463},
MRREVIEWER = {R.\ F.\ Lax},
       URL = {http://projecteuclid.org/euclid.jdg/1214447001},
}

@article {CHS08,
    AUTHOR = {Coskun, I. and Harris, J. and Starr, J.},
     TITLE = {The effective cone of the {K}ontsevich moduli space},
   JOURNAL = {Canad. Math. Bull.},
  FJOURNAL = {Canadian Mathematical Bulletin. Bulletin Canadien de
              Math\'ematiques},
    VOLUME = {51},
      YEAR = {2008},
    NUMBER = {4},
     PAGES = {519--534},
      ISSN = {0008-4395,1496-4287},
   MRCLASS = {14D20 (14C20 14H10)},
  MRNUMBER = {2462457},
MRREVIEWER = {Montserrat\ Teixidor i Bigas},
       DOI = {10.4153/CMB-2008-052-5},
       URL = {https://doi.org/10.4153/CMB-2008-052-5},
}

@misc{weilrep,
  author = {Williams, B.},
  title = {{W}eil{R}ep},
  note = {Available at \url{https://github.com/btw-47/weilrep}},
  url = {https://github.com/btw-47/weilrep},
}

@article {FP05,
    AUTHOR = {Farkas, G. and Popa, M.},
     TITLE = {Effective divisors on $\overline{\mathcal{M}}_g$, curves on
              {$K3$} surfaces, and the slope conjecture},
   JOURNAL = {J. Algebraic Geom.},
  FJOURNAL = {Journal of Algebraic Geometry},
    VOLUME = {14},
      YEAR = {2005},
    NUMBER = {2},
     PAGES = {241--267},
      ISSN = {1056-3911,1534-7486},
   MRCLASS = {14H10 (14H60 14J28)},
  MRNUMBER = {2123229},
MRREVIEWER = {I.\ Dolgachev},
       DOI = {10.1090/S1056-3911-04-00392-3},
       URL = {https://doi.org/10.1090/S1056-3911-04-00392-3},
}

@article {FGSMV14,
    AUTHOR = {Farkas, G. and Grushevsky, S. and Salvati Manni, R. and Verra,
              A.},
     TITLE = {Singularities of theta divisors and the geometry of
              {$\mathcal{A}_5$}},
   JOURNAL = {J. Eur. Math. Soc. (JEMS)},
  FJOURNAL = {Journal of the European Mathematical Society (JEMS)},
    VOLUME = {16},
      YEAR = {2014},
    NUMBER = {9},
     PAGES = {1817--1848},
      ISSN = {1435-9855,1435-9863},
   MRCLASS = {14K10 (14H10 14H40 14K25)},
  MRNUMBER = {3273309},
MRREVIEWER = {Dan\ Petersen},
       DOI = {10.4171/JEMS/476},
       URL = {https://doi.org/10.4171/JEMS/476},
}

@book {Ful98,
    AUTHOR = {Fulton, W.},
     TITLE = {Intersection theory},
    SERIES = {Ergebnisse der Mathematik und ihrer Grenzgebiete. 3. Folge. A
              Series of Modern Surveys in Mathematics [Results in
              Mathematics and Related Areas. 3rd Series. A Series of Modern
              Surveys in Mathematics]},
    VOLUME = {2},
   EDITION = {Second},
 PUBLISHER = {Springer-Verlag, Berlin},
      YEAR = {1998},
     PAGES = {xiv+470},
      ISBN = {3-540-62046-X; 0-387-98549-2},
   MRCLASS = {14C17 (14-02)},
       DOI = {10.1007/978-1-4612-1700-8},
       URL = {https://doi.org/10.1007/978-1-4612-1700-8},
}

@article {GHS07,
    AUTHOR = {Gritsenko, V. A. and Hulek, K. and Sankaran, G. K.},
     TITLE = {The {K}odaira dimension of the moduli of {$K3$} surfaces},
   JOURNAL = {Invent. Math.},
  FJOURNAL = {Inventiones Mathematicae},
    VOLUME = {169},
      YEAR = {2007},
    NUMBER = {3},
     PAGES = {519--567},
      ISSN = {0020-9910},
   MRCLASS = {14J28 (14J10 32M15 32N15)},
MRREVIEWER = {I. Dolgachev},
       DOI = {10.1007/s00222-007-0054-1},
       URL = {https://doi.org/10.1007/s00222-007-0054-1},
}

@incollection {GHS13,
    AUTHOR = {Gritsenko, V. and Hulek, K. and Sankaran, G. K.},
     TITLE = {Moduli of {K}3 surfaces and irreducible symplectic manifolds},
 BOOKTITLE = {Handbook of moduli. {V}ol. {I}},
    SERIES = {Adv. Lect. Math. (ALM)},
    VOLUME = {24},
     PAGES = {459--526},
 PUBLISHER = {Int. Press, Somerville, MA},
      YEAR = {2013},
   MRCLASS = {14J10 (14J28 53C26)},
MRREVIEWER = {Aristides I. Kontogeorgis},
}

@article {HM90,
    AUTHOR = {Harris, J. and Morrison, I.},
     TITLE = {Slopes of effective divisors on the moduli space of stable
              curves},
   JOURNAL = {Invent. Math.},
  FJOURNAL = {Inventiones Mathematicae},
    VOLUME = {99},
      YEAR = {1990},
    NUMBER = {2},
     PAGES = {321--355},
      ISSN = {0020-9910,1432-1297},
   MRCLASS = {14H10},
  MRNUMBER = {1031904},
MRREVIEWER = {Mei\ Chu\ Chang},
       DOI = {10.1007/BF01234422},
       URL = {https://doi.org/10.1007/BF01234422},
}

@book {Huy16,
    AUTHOR = {Huybrechts, D.},
     TITLE = {Lectures on {K}3 surfaces},
    SERIES = {Cambridge Studies in Advanced Mathematics},
    VOLUME = {158},
 PUBLISHER = {Cambridge University Press, Cambridge},
      YEAR = {2016},
     PAGES = {xi+485},
      ISBN = {978-1-107-15304-2},
   MRCLASS = {14J28},
       DOI = {10.1017/CBO9781316594193},
       URL = {https://doi.org/10.1017/CBO9781316594193},
}

@article {Kon88,
    AUTHOR = {Kondo, S.},
     TITLE = {On the {A}lbanese variety of the moduli space of polarized
              {$K3$} surfaces},
   JOURNAL = {Invent. Math.},
  FJOURNAL = {Inventiones Mathematicae},
    VOLUME = {91},
      YEAR = {1988},
    NUMBER = {3},
     PAGES = {587--593},
      ISSN = {0020-9910,1432-1297},
   MRCLASS = {32G13 (14J15 14J28 32J05 32J25 32M15)},
  MRNUMBER = {928499},
MRREVIEWER = {Marcus\ Wright},
       DOI = {10.1007/BF01388788},
       URL = {https://doi.org/10.1007/BF01388788},
}

@article {Kud03,
    AUTHOR = {Kudla, S.S.},
     TITLE = {Integrals of {B}orcherds forms},
   JOURNAL = {Compositio Math.},
  FJOURNAL = {Compositio Mathematica},
    VOLUME = {137},
      YEAR = {2003},
    NUMBER = {3},
     PAGES = {293--349},
      ISSN = {0010-437X,1570-5846},
   MRCLASS = {11F30 (11F55 11G18 11M36 14G35)},
       DOI = {10.1023/A:1024127100993},
       URL = {https://doi.org/10.1023/A:1024127100993},
}

@incollection {LO18,
    AUTHOR = {Laza, R. and O'Grady, K. G.},
     TITLE = {G{IT} versus {B}aily-{B}orel compactification for quartic
              {K3} surfaces},
 BOOKTITLE = {Geometry of moduli},
    SERIES = {Abel Symp.},
    VOLUME = {14},
     PAGES = {217--283},
 PUBLISHER = {Springer, Cham},
      YEAR = {2018},
   MRCLASS = {14L24 (14J10 14J28)},
MRREVIEWER = {I. Dolgachev},
}

@article {LO19,
    AUTHOR = {Laza, R. and O'Grady, K.G.},
     TITLE = {Birational geometry of the moduli space of quartic {$K3$}
              surfaces},
   JOURNAL = {Compos. Math.},
  FJOURNAL = {Compositio Mathematica},
    VOLUME = {155},
      YEAR = {2019},
    NUMBER = {9},
     PAGES = {1655--1710},
      ISSN = {0010-437X},
   MRCLASS = {14J10 (14D20 14J28 14L24)},
MRREVIEWER = {Laura Pertusi},
       DOI = {10.1112/s0010437x19007516},
       URL = {https://doi.org/10.1112/s0010437x19007516},
}

@inproceedings{Loo86,
    AUTHOR = {Looijenga, E.},
     TITLE = {New compactifications of locally symmetric varieties},
 BOOKTITLE = {Proceedings of the 1984 {V}ancouver conference in algebraic
              geometry},
    SERIES = {CMS Conf. Proc.},
    VOLUME = {6},
     PAGES = {341--364},
 PUBLISHER = {Amer. Math. Soc., Providence, RI},
      YEAR = {1986},
      ISBN = {0-8218-6010-0},
   MRCLASS = {32M15 (14J15 14J28 32G13 32J05)},
MRREVIEWER = {Yukihiko\ Namikawa},
       DOI = {10.1016/0040-9383(86)90045-5},
       URL = {https://doi.org/10.1016/0040-9383(86)90045-5},
}

@incollection {MP13,
    AUTHOR = {Maulik, D. and Pandharipande, R.},
     TITLE = {Gromov-{W}itten theory and {N}oether-{L}efschetz theory},
 BOOKTITLE = {A celebration of algebraic geometry},
    SERIES = {Clay Math. Proc.},
    VOLUME = {18},
     PAGES = {469--507},
 PUBLISHER = {Amer. Math. Soc., Providence, RI},
      YEAR = {2013},
   MRCLASS = {14N35 (14J10 14J28)},
MRREVIEWER = {Satoshi Minabe},
}

@incollection {Muk88,
    AUTHOR = {Mukai, S.},
     TITLE = {Curves, {$K3$} surfaces and {F}ano {$3$}-folds of genus {$\leq
              10$}},
 BOOKTITLE = {Algebraic geometry and commutative algebra, {V}ol. {I}},
     PAGES = {357--377},
 PUBLISHER = {Kinokuniya, Tokyo},
      YEAR = {1988},
   MRCLASS = {14J10 (14J28 14J30 32G13 32J15 32M10)},
MRREVIEWER = {Peter Nielsen},
}

@incollection {Muk92,
    AUTHOR = {Mukai, S.},
     TITLE = {Polarized {$K3$} surfaces of genus {$18$} and {$20$}},
 BOOKTITLE = {Complex projective geometry ({T}rieste, 1989/{B}ergen, 1989)},
    SERIES = {London Math. Soc. Lecture Note Ser.},
    VOLUME = {179},
     PAGES = {264--276},
 PUBLISHER = {Cambridge Univ. Press, Cambridge},
      YEAR = {1992},
   MRCLASS = {14J28 (14J10 32J15)},
MRREVIEWER = {Shigeyuki Kondo},
       DOI = {10.1017/CBO9780511662652.019},
       URL = {https://doi.org/10.1017/CBO9780511662652.019},
}

@incollection {Muk06,
    AUTHOR = {Mukai, S.},
     TITLE = {Polarized {$K3$} surfaces of genus thirteen},
 BOOKTITLE = {Moduli spaces and arithmetic geometry},
    SERIES = {Adv. Stud. Pure Math.},
    VOLUME = {45},
     PAGES = {315--326},
 PUBLISHER = {Math. Soc. Japan, Tokyo},
      YEAR = {2006},
   MRCLASS = {14J28},
MRREVIEWER = {Alex Degtyarev},
       DOI = {10.2969/aspm/04510315},
       URL = {https://doi.org/10.2969/aspm/04510315},
}

@article {Muk10,
    AUTHOR = {Mukai, S.},
     TITLE = {Curves and symmetric spaces, {II}},
   JOURNAL = {Ann. of Math. (2)},
  FJOURNAL = {Annals of Mathematics. Second Series},
    VOLUME = {172},
      YEAR = {2010},
    NUMBER = {3},
     PAGES = {1539--1558},
      ISSN = {0003-486X},
   MRCLASS = {14H45 (14C20 14H51 14M15)},
MRREVIEWER = {Raquel Mallavibarrena},
       DOI = {10.4007/annals.2010.172.1539},
       URL = {https://doi.org/10.4007/annals.2010.172.1539},
}

@incollection {Muk16,
    AUTHOR = {Mukai, S.},
     TITLE = {K3 surfaces of genus sixteen},
 BOOKTITLE = {Minimal models and extremal rays ({K}yoto, 2011)},
    SERIES = {Adv. Stud. Pure Math.},
    VOLUME = {70},
     PAGES = {379--396},
 PUBLISHER = {Math. Soc. Japan, [Tokyo]},
      YEAR = {2016},
   MRCLASS = {14J28 (14M20)},
MRREVIEWER = {Zhiyu Tian},
       DOI = {10.2969/aspm/07010379},
       URL = {https://doi.org/10.2969/aspm/07010379},
}

@article {Mum67,
    AUTHOR = {Mumford, D.},
     TITLE = {Abelian quotients of the {T}eichm\"uller modular group},
   JOURNAL = {J. Analyse Math.},
  FJOURNAL = {Journal d'Analyse Math\'ematique},
    VOLUME = {18},
      YEAR = {1967},
     PAGES = {227--244},
      ISSN = {0021-7670,1565-8538},
   MRCLASS = {14.51 (32.00)},
  MRNUMBER = {219543},
MRREVIEWER = {R.\ C.\ Gunning},
       DOI = {10.1007/BF02798046},
       URL = {https://doi.org/10.1007/BF02798046},
}

@article {Mum77,
    AUTHOR = {Mumford, D.},
     TITLE = {Hirzebruch's proportionality theorem in the noncompact case},
   JOURNAL = {Invent. Math.},
  FJOURNAL = {Inventiones Mathematicae},
    VOLUME = {42},
      YEAR = {1977},
     PAGES = {239--272},
      ISSN = {0020-9910},
   MRCLASS = {32M15 (14D20 22E40 32N10 57R20)},
MRREVIEWER = {A. L. Onishchik},
       DOI = {10.1007/BF01389790},
       URL = {https://doi.org/10.1007/BF01389790},
}

@phdthesis{Pet15,
    title    = {Modular forms on the moduli space of polarised {K}3 surfaces},
    school   = {University of Amsterdam},
    author   = {Peterson, A.},
    year     = {2015},
    note={Available at \url{http://hdl.handle.net/ 11245/2.162072}},
}

@article {Rei76,
    AUTHOR = {Reid, M.},
     TITLE = {Hyperelliptic linear systems on a {K}3 surface},
   JOURNAL = {J. London Math. Soc. (2)},
  FJOURNAL = {Journal of the London Mathematical Society. Second Series},
    VOLUME = {13},
      YEAR = {1976},
    NUMBER = {3},
     PAGES = {427--437},
      ISSN = {0024-6107,1469-7750},
   MRCLASS = {14J25 (14C20)},
MRREVIEWER = {Daniel\ Comenetz},
       DOI = {10.1112/jlms/s2-13.3.427},
       URL = {https://doi.org/10.1112/jlms/s2-13.3.427},
}

@book {Rul01,
    AUTHOR = {Rulla, W. F.},
     TITLE = {The birational geometry of moduli space {M}(3) and moduli
              space {M}(2,1)},
      NOTE = {Thesis (Ph.D.)--The University of Texas at Austin},
 PUBLISHER = {ProQuest LLC, Ann Arbor, MI},
      YEAR = {2001},
     PAGES = {188},
      ISBN = {978-0493-18174-5},
   MRCLASS = {99-05},
       URL =
              {http://gateway.proquest.com/openurl?url_ver=Z39.88-2004&rft_val_fmt=info:ofi/fmt:kev:mtx:dissertation&res_dat=xri:pqdiss&rft_dat=xri:pqdiss:3008434},
}

@online{sagemath,
  Key          = {SageMath},
  Author       = {{Sage Developers}},
  Title        = {{S}ageMath, the {S}age {M}athematics {S}oftware {S}ystem ({V}ersion 10.5)},
  url         = {\https://www.sagemath.org},
  Year         = {2024},
}

@incollection {SM92,
    AUTHOR = {Salvati Manni, Riccardo},
     TITLE = {Modular forms of the fourth degree. {R}emark on the paper:
              ``{S}lopes of effective divisors on the moduli space of stable
              curves'' [{I}nvent.\ {M}ath.\ {\bf 99} (1990), no. 2,
              321--355; {MR}1031904 (91d:14009)] by {J}. {D}. {H}arris and
              {I}. {M}orrison},
 BOOKTITLE = {Classification of irregular varieties ({T}rento, 1990)},
    SERIES = {Lecture Notes in Math.},
    VOLUME = {1515},
     PAGES = {106--111},
 PUBLISHER = {Springer, Berlin},
      YEAR = {1992},
      ISBN = {3-540-55295-2},
   MRCLASS = {14H15 (11F46 14H10)},
MRREVIEWER = {Ziv\ Ran},
       DOI = {10.1007/BFb0098340},
       URL = {https://doi.org/10.1007/BFb0098340},
}

@article {Sha80,
    AUTHOR = {Shah, J.},
     TITLE = {A complete moduli space for {$K3$}\ surfaces of degree {$2$}},
   JOURNAL = {Ann. of Math. (2)},
  FJOURNAL = {Annals of Mathematics. Second Series},
    VOLUME = {112},
      YEAR = {1980},
    NUMBER = {3},
     PAGES = {485--510},
      ISSN = {0003-486X},
   MRCLASS = {14J10 (14D25 14J25)},
  MRNUMBER = {595204},
MRREVIEWER = {Miles\ Reid},
       DOI = {10.2307/1971089},
       URL = {https://doi.org/10.2307/1971089},
}

@article {Tan98,
    AUTHOR = {Tan, S.-L.},
     TITLE = {On the slopes of the moduli spaces of curves},
   JOURNAL = {Internat. J. Math.},
  FJOURNAL = {International Journal of Mathematics},
    VOLUME = {9},
      YEAR = {1998},
    NUMBER = {1},
     PAGES = {119--127},
      ISSN = {0129-167X,1793-6519},
   MRCLASS = {14H10},
  MRNUMBER = {1612259},
MRREVIEWER = {Montserrat\ Teixidor i Bigas},
       DOI = {10.1142/S0129167X98000087},
       URL = {https://doi.org/10.1142/S0129167X98000087},
}

@book {Voi14,
    AUTHOR = {Voisin, C.},
     TITLE = {Chow rings, decomposition of the diagonal, and the topology of
              families},
    SERIES = {Annals of Mathematics Studies},
    VOLUME = {187},
 PUBLISHER = {Princeton University Press, Princeton, NJ},
      YEAR = {2014},
     PAGES = {viii+163},
      ISBN = {978-0-691-16051-1; 0-691-16051-1},
   MRCLASS = {14-02 (14C15 14C25 14C30 32C18)},
       DOI = {10.1515/9781400850532},
       URL = {https://doi.org/10.1515/9781400850532},
}

@online {zh;phd,
    AUTHOR = {Zhang, W.},
     TITLE = {Modularity of Generating Functions of
Special Cycles on Shimura Varieties},
      Note = {PhD Thesis},
      YEAR = {2009},
       URL = {http://www.math.columbia.edu/~thaddeus/theses/2009/zhang.pdf},
}

@article {zuffetti-cod2,
    AUTHOR = {Zuffetti, Riccardo},
     TITLE = {Cones of special cycles of codimension 2 on orthogonal
              {S}himura varieties},
   JOURNAL = {Trans. Amer. Math. Soc.},
  FJOURNAL = {Transactions of the American Mathematical Society},
    VOLUME = {375},
      YEAR = {2022},
    NUMBER = {10},
     PAGES = {7385--7441},
      ISSN = {0002-9947,1088-6850},
   MRCLASS = {14G35 (11F30 11F46 11G18 14C25)},
  MRNUMBER = {4491430},
MRREVIEWER = {Amir\ D\v zambi\'c},
       DOI = {10.1090/tran/8757},
       URL = {https://doi.org/10.1090/tran/8757},
}

@article {zuffetti-equidistribution,
    AUTHOR = {Zuffetti, Riccardo},
     TITLE = {Cones of orthogonal {S}himura subvarieties and
              equidistribution},
   JOURNAL = {Manuscripta Math.},
  FJOURNAL = {Manuscripta Mathematica},
    VOLUME = {175},
      YEAR = {2024},
    NUMBER = {3-4},
     PAGES = {791--811},
      ISSN = {0025-2611,1432-1785},
   MRCLASS = {14G35 (11F27 11G18 14C25 14F40)},
  MRNUMBER = {4822033},
       DOI = {10.1007/s00229-024-01586-8},
       URL = {https://doi.org/10.1007/s00229-024-01586-8},
}


\appendix
\section{Tables of extremal rays for unimodular and K3 lattices}

In this section, we record, for various choices of lattice $\Lambda$, the primitive Heegner divisors~$P_{-m, \mu}$ that we have verified using WeilRep satisfy the hypotheses of Theorem~\ref{thm: extremality2} and therefore are extremal. 

\begin{table}[!ht]
\begin{tabular}{| c | c |} 
\hline
$r$ such that $\Lambda=U^{\oplus 2}\oplus E_8(-1)^{\oplus r}$ & $P_m\coloneqq P_{m,0}$ extremal \\
\hline
$4\leq r\leq 7$ & $P_{-1}$ \\
\hline
$8 \leq r \leq 12$ & $P_{-1}$, $P_{-2}$ \\
\hline
$13 \leq r \leq 16$ & $P_{-1}$, $P_{-2}$, $P_{-3}$ \\
\hline
$17 \leq r \leq 20$ & $P_{-1}$, $P_{-2}$, $P_{-3}$, $P_{-4}$ \\
\hline
$21 \leq r \leq 24$ & $P_{-1}$, $P_{-2}$, $P_{-3}$, $P_{-4}$, $P_{-5}$ \\
\hline
\end{tabular}
\caption{Examples of extremal rays for $\Lambda$ unimodular.}
\label{tableoutmethod2:1}
\end{table}

\begin{table}[!ht]
\begin{tabular}{| c | c | c | c |}
\hline
 $d$ & $\dim {\rm{Pic}}_{\mathbb{Q}}\left(\mathcal{F}_{2d}\right)$ & $P_{m,\mu}$ extremal \\
 \hline
 $d=1$ & 2 & $P_{-1/4,\ell_*}$ \\ 
 \hline
 $d=2$ & 3 & $P_{-1/8,\ell_*}$, $P_{-1/2,2\ell_*}$,  \\ 
 \hline
  $d=3$ & 4 & $P_{-1/12,\ell_*}$, $P_{-1/3,2\ell_*}$  \\ 
 \hline
  $d=4$ & 4 & $P_{-1/16,\ell_*}$, $P_{-9/16,3\ell_*}$, $P_{-1/4,2\ell_*}$  \\ 
 \hline
  $d=5$ & 6 & $P_{-1/20,\ell_*}$, $P_{-9/20,3\ell_*}$, $P_{-1/5,2\ell_*}$,  \\ 
 \hline
  $d=6$ & 7 & $P_{-1/24,\ell_*}$, $P_{-1/6,2\ell_*}$, $P_{-3/8,3\ell_*}$, $P_{-1/24,5\ell_*}$, $P_{-1/2,6\ell_*}$ \\ 
  \hline
  $d=7$ & 7 & $P_{-1/28,\ell_*}$, $P_{-1/7,2\ell_*}$, $P_{-9/28,3\ell_*}$, $P_{-4/7,4\ell_*}$, $P_{-2/7, 6\ell_*}$ \\ 
  \hline
  $d=8$ & 8 & $P_{-1/32,\ell_*}$, $P_{-1/8,2\ell_*}$, $P_{-9/32,3\ell_*}$, $P_{-1/2,4\ell_*}$, $P_{-1/8,6\ell_*}$, $P_{-17/32,7\ell_*}$ \\ 
 \hline
 $d=9$ & 9 & $P_{-1/36,\ell_*}$, $P_{-1/9,2\ell_*}$, $P_{-1/4,3\ell_*}$, $P_{-4/9,4\ell_*}$, $P_{-13/36,7\ell_*}$, $P_{-1/4,9\ell_*}$ \\ 
 \hline
 $d=10$ & 10 & \makecell{$P_{-1/40,\ell_*}$, $P_{-1/10,2\ell_*}$, $P_{-9/40,3\ell_*}$, $P_{-2/5,4\ell_*}$, $P_{-5/8,5\ell_*}$, \\ $P_{-9/40,7\ell_*}$, $P_{-3/5,8\ell_*}$, $P_{-1/40,9\ell_*}$, $P_{-1/2,10\ell_*}$} \\ 
 \hline
 $d=11$ & 11 & \makecell{$P_{-1/44,\ell_*}$, $P_{-1/11,2\ell_*}$, $P_{-9/44,3\ell_*}$, $P_{-4/11,4\ell_*}$, $P_{-25/44,5\ell_*}$, \\ $P_{-5/44,7\ell_*}$, $P_{-5/11,8\ell_*}$, $P_{-3/11,10\ell_*}$} \\ 
 \hline
  $d=12$ & 12 & \makecell{$P_{-1/48,\ell_*}$, $P_{-1/12,2\ell_*}$, $P_{-3/16,3\ell_*}$, $P_{-1/3,4\ell_*}$, $P_{-25/48,5\ell_*}$, \\$P_{-1/48,7\ell_*}$, $P_{-1/3,8\ell_*}$, $P_{-11/16,9\ell_*}$, $P_{-1/12,10\ell_*}$, $P_{-25/48,11\ell_*}$} \\
  \hline
  $d=13$ & 12 & \makecell{$P_{-1/52,\ell_*}$, $P_{-1/13,2\ell_*}$, $P_{-9/52,3\ell_*}$, $P_{-4/13,4\ell_*}$, $P_{-25/52,5\ell_*}$,\\ $P_{-3/13,8\ell_*}$,$P_{-29/52,9\ell_*}$, $P_{-17/52,11\ell_*}$, $P_{-1/4,13\ell_*}$ }\\
  \hline
  $d=14$ & 14 & \makecell{$P_{-1/56,\ell_*}$, $P_{-1/14,2\ell_*}$, $P_{-9/56,3\ell_*}$, $P_{-2/7,4\ell_*}$, $P_{-25/56,5\ell_*}$,\\ $P_{-9/14,6\ell_*}$, $P_{-1/7,8\ell_*}$, $P_{-25/56,9\ell_*}$, $P_{-9/56,11\ell_*}$, $P_{-4/7,12\ell_*}$\\ $P_{-1/56,13\ell_*}$, $P_{-1/2,14\ell_*}$}\\
  \hline
  $d=15$ & 15 & \makecell{$P_{-1/60,\ell_*}$, $P_{-1/15,2\ell_*}$, $P_{-3/20,3\ell_*}$, $P_{-4/15,4\ell_*}$, $P_{-5/12,5\ell_*}$,\\ $P_{-3/5,6\ell_*}$, $P_{-1/15,8\ell_*}$, $P_{-7/20,9\ell_*}$, $P_{-2/3,10\ell_*}$, $P_{-1/60,11\ell_*}$,\\ $P_{-2/5,12\ell_*}$, $P_{-4/15,14\ell_*}$}\\
  \hline
  $d=16$ & 14 & \makecell{$P_{-1/64,\ell_*}$, $P_{-1/16,2\ell_*}$, $P_{-9/64,3\ell_*}$, $P_{-1/4,4\ell_*}$, $P_{-25/64,5\ell_*}$,\\
  $P_{-9/16,6\ell_*}$, $P_{-17/64,9\ell_*}$, $P_{-9/16,10\ell_*}$, $P_{-1/4,12\ell_*}$, $P_{-41/64,13\ell_*}$,\\ $P_{-1/16,14\ell_*}$, $P_{-33/64,15\ell_*}$ }\\
  \hline
  $d=17$ & 16 & \makecell{$P_{-1/68,\ell_*}$, $P_{-1/17,2\ell_*}$, $P_{-9/68,3\ell_*}$, $P_{-4/17,4\ell_*}$, $P_{-25/68,5\ell_*}$,\\ $P_{-9/17,6\ell_*}$, $P_{-13/68,9\ell_*}$, $P_{-8/17,10\ell_*}$, $P_{-2/17,12\ell_*}$, $P_{-33/68,13\ell_*}$, \\$P_{-21/68,15\ell_*}$, $P_{-1/4,17\ell_*}$, }\\
  \hline
  $d=18$ & 17 & \makecell{$P_{-1/72,\ell_*}$, $P_{-1/18,2\ell_*}$, $P_{-1/8,3\ell_*}$, $P_{-2/9,4\ell_*}$, $P_{-25/72,5\ell_*}$, \\$P_{-1/2,6\ell_*}$, $P_{-49/72,7\ell_*}$, $P_{-1/8,9\ell_*}$, $P_{-7/18,10\ell_*}$, $P_{-49/72,11\ell_*}$, \\$P_{-25/72,13\ell_*}$, $P_{-1/8,15\ell_*}$, $P_{-5/9, 16\ell_*}$, $P_{-1/72,17\ell_*}$, $P_{-1/2,18\ell_*}$}\\
  \hline
  $d=19$ & 17 & \makecell{$P_{-1/76,\ell_*}$, $P_{-1/19,2\ell_*}$, $P_{-9/76,3\ell_*}$, $P_{-4/19,4\ell_*}$, $P_{-25/76,5\ell_*}$,\\ $P_{-9/19,6\ell_*}$, $P_{-49/76,7\ell_*}$, $P_{-5/76,9\ell_*}$, $P_{-6/19,10\ell_*}$, $P_{-45/76,11\ell_*}$, \\$P_{-17/76,13\ell_*}$, $P_{-11/19,14\ell_*}$, $P_{-7/19,16\ell_*}$, $P_{-5/19,18\ell_*}$, $P_{-3/4,19\ell_*}$}\\
  \hline
  $d=20$ & 19 & \makecell{$P_{-1/80,\ell_*}$, $P_{-1/20,2\ell_*}$, $P_{-9/80,3\ell_*}$, $P_{-1/5,4\ell_*}$, $P_{-5/16,5\ell_*}$, \\$P_{-9/20,6\ell_*}$, $P_{-49/80, 7\ell_*}$, $P_{-1/80,9\ell_*}$, $P_{-1/4,10\ell_*}$, $P_{-41/80, 11\ell_*}$, \\ $P_{-9/80,13\ell_*}$, $P_{-9/20, 14\ell_*}$, $P_{-1/5,16\ell_*}$, $P_{-49/80,17\ell_*}$, $P_{-1/20, 18\ell_*}$, \\ $P_{-41/80, 19\ell_*}$}\\
  \hline
\end{tabular}
\caption{Examples of extremal rays on moduli spaces of K3 surfaces $\mathcal{F}_{2d}$.}
\label{tableoutmethod2:2}
\end{table}

\end{document}